\documentclass{siamart1116}

\usepackage{lipsum}
\usepackage{amsfonts}
\usepackage{graphicx}
\usepackage{epstopdf}
\usepackage{algorithmic}
\usepackage{amsmath}
\usepackage{amssymb}
\usepackage{bm}
\usepackage{mathbbold}
\usepackage{mathrsfs}
\usepackage{multirow}
\usepackage{dcolumn}
\usepackage{listings}
\usepackage{subfigure}
\usepackage{enumerate}
\ifpdf
  \DeclareGraphicsExtensions{.eps,.pdf,.png,.jpg}
\else
  \DeclareGraphicsExtensions{.eps}
\fi

\numberwithin{theorem}{section}

\newcommand{\ud}{\mathrm{d}}

\newtheorem{remark}{Remark~}[section]

\newcommand{\TheTitle}{Stochastic gradient-free descents}
\newcommand{\TheTitleShort}{Stochastic gradient-free descents}
\newcommand{\TheAuthors}{X. Luo and X. Xu}

\headers{\TheTitleShort}{\TheAuthors}

\title{{\TheTitle}}

\author{
  Xiaopeng Luo\thanks{Department of Chemistry, Princeton University, Princeton, NJ 08544, and Department of Control and Systems Engineering, School of Management and Engineering, Nanjing University, Nanjing, 210008, China (\email{luo.permenant@gmail.com}, \email{xu.permenant@gmail.com}.)}
  \and
  Xin Xu\footnotemark[1]
}

\usepackage{amsopn}

\ifpdf
\hypersetup{
  pdftitle={\TheTitle},
  pdfauthor={\TheAuthors}
}
\fi

\newsiamthm{assumption}{assumption}
\crefname{assumption}{Assumption}{Assumptions}

\begin{document}

\maketitle

\begin{abstract}
  In this paper we propose stochastic gradient-free methods and accelerated methods with momentum for solving stochastic optimization problems. All these methods rely on stochastic directions rather than stochastic gradients. We analyze the convergence behavior of these methods under the mean-variance framework, and also provide a theoretical analysis about the inclusion of momentum in stochastic settings which reveals that the momentum term we used adds a deviation of order $\mathcal{O}(1/k)$ but controls the variance at the order $\mathcal{O}(1/k)$ for the $k$th iteration. So it is shown that, when employing a decaying stepsize $\alpha_k=\mathcal{O}(1/k)$, the stochastic gradient-free methods can still maintain the sublinear convergence rate $\mathcal{O}(1/k)$ and the accelerated methods with momentum can achieve a convergence rate $\mathcal{O}(1/k^2)$ in probability for the strongly convex objectives with Lipschitz gradients; and all these methods converge to a solution with a zero expected gradient norm when the objective function is nonconvex, twice differentiable and bounded below.
\end{abstract}



\section{Introduction}
\label{SGFD:s1}

In this paper, we develop and analyze stochastic \emph{gradient-free} descent (SGFD) methods and accelerated methods with momentum for solving the following class of stochastic optimization problems:
\begin{equation}\label{SGFD:eq:OP}
  x_*=\arg\min_{x\in\mathbb{R}^d}F(x),
\end{equation}
where the real-valued function $F$ is defined by
\begin{equation}\label{SGFD:eq:f}
  F(x):=\mathbb{E}_\xi\big[f(x,\xi)\big]=\int_\Xi f(x,\xi)\ud P(\xi),
\end{equation}
and $\{f(\cdot,\xi),\xi\in\Xi\}$ be a collection of real-valued functions with a given probability distribution $P$ over the index set $\Xi$. 

In the context of machine learning applications, $f(x,\xi)$ is often treated as the loss function of a prediction function $h$ incurred by the parameter vector $x$ with respect to the randomly selected sample $\xi$ from a sample set $\{(\mathcal{X}_i,\mathcal{Y}_i)\}_{i=1}^n$, i.e., $\ell(h(\mathcal{X}_i;x),\mathcal{Y}_i)$; accordingly, $F$ is treated as the empirical risk given a parameter vector $x$ with respect to the distribution $P$. The popular methodology for such problems is the stochastic gradient (SG) method \cite{ZinkevichM2003A_SG,ZhangT2004A_SG,BottouL2007A_TradeoffsLearning,
NemirovskiA2009A_StochasticProgramming,ShwartzS2011A_subgradientSVM}. Specifically, with an initial point $x_1$, these methods are characterized by the iteration
\begin{equation}\label{SGFD:eq:GD}
  x_{k+1}=x_k-\alpha_kg(x_k,\xi_k)
\end{equation}
where $\alpha_k>0$ is the stepsize and $g(x_k,\xi_k)$ is the stochastic gradient defined by
\begin{equation}\label{SGFD:eq:s}
\begin{aligned}
  g(x_k,\xi_k)=\left\{
  \begin{array}{c}
    \nabla f(x_k,\xi_k), \\[0.5em]
    \frac{1}{n_k}\sum_{i=1}^{n_k}\nabla f(x_k,\xi_{k,i}),
  \end{array}\right.
\end{aligned}
\end{equation}
which is an unbiased estimator of the socalled full gradient $\nabla F(x_k)$ \cite{ShapiroA2009M_StochasticProgramming,BottouL2018R_SGD}.

The SG method was originally developed by Robbins and Monro \cite{RobbinsH1951A_SG} for smooth stochastic approximation problems. It has convergence guarantees \cite{ChungKL1954A_SG,SacksJ1958A_SG,
NemirovskiA2009A_StochasticProgramming,ZinkevichM2003A_SG,BottouL2018R_SGD} and has gained extensive empirical success in large-scale convex and nonconvex stochastic optimization \cite{ZhangT2004A_SG,BottouL2007A_TradeoffsLearning,
ShwartzS2011A_subgradientSVM,DeanJ2012A_SG&DNN,LeCunY2015A_DeepLearning}. However, there are still notable difficulties with the SG method \cite{AsiH2019A_SPP}, and some of them are closely related to the gradient itself. For example, it might cause the vanishing and exploding gradient in training artificial neural networks \cite{BengioY1994A_GradientDifficult,
PascanuR2013A_DifficultyRNN}; moreover, the gradient is sometimes very difficult or even impossible to obtain.

From this point it is worth considering a simpler strategy that could avoid direct gradient evaluations. Actually, gradient-free optimization methods, which is known in the literature as the derivative-free \cite{ConnA2009M_DerivativeFree,StichS2013A_ConvexDF,
NesterovY2017A_GradientFree,GorbunovE2019A_DerivativeFree} or zero-order methods \cite{DuchiJ2015A_ZeroOrderCO,ShamirO2017A_ZeroOrderConvex} and also called bandit optimization in the machine learning literature \cite{AgarwalA2013A_Bandit,HazanE2014A_BanditOptimization,
ShamirO2017A_ZeroOrderConvex}, were among the first schemes suggested in the early days of the development of optimization theory \cite{MatyasJ1965A_RandomOptimization}; unfortunately, these methods seem to be much more difficult for theoretical investigation and the possible rate of convergence is far below the efficiency of the gradient-based schemes on an empirical level \cite{NesterovY2017A_GradientFree}. However, we see a restoration of the interest to this topic in the last years \cite{ConnA2009M_DerivativeFree,NesterovY2017A_GradientFree} and the gradient-free schemes have been generalized to solve stochastic optimization problems; e.g., see \cite{AgarwalA2013A_Bandit,HazanE2014A_BanditOptimization,
DuchiJ2015A_ZeroOrderCO,ShamirO2017A_ZeroOrderConvex,
NesterovY2017A_GradientFree}. And these theoretical analyses ensure a rate of convergence $\mathcal{O}(1/\sqrt{k})$ for convex functions in stochastic settings, e.g., see \cite{GorbunovE2019A_DerivativeFree} for recent work. But until now, there have been no reports showing a convergence rate $\mathcal{O}(1/k)$ or beyond under strong convex conditions for stochastic optimization problems. 

In this work, we first propose stochastic gradient-free methods with the expected rate of convergence $\mathcal{O}(1/k)$ in probability for strongly convex objectives with Lipschitz gradients, which retain the main advantage of gradient-based methods \cite{ChungKL1954A_SG,SacksJ1958A_SG,
NemirovskiA2009A_StochasticProgramming,BottouL2018R_SGD}. We use the mean-variance framework for our analysis, which enables us to understand the roles of deviation and variance in the error bound. In addition, we establish an analysis based on the gamma function to obtain the orders for terms in the error bound when employing a decaying stepsize $\alpha_k=\mathcal{O}(1/k)$ under a strong convex condition.

These advantages allow us to design an appropriate momentum term to control both the additional deviation and variance at the order $\mathcal{O}(1/k)$. And then, by using such a momentum term, we propose accelerated methods with the expected rate of convergence $\mathcal{O}(1/k^2)$ in probability for strongly convex objectives with Lipschitz gradients, which is similar to the socalled heavy ball method in deterministic settings \cite{PolyakB1964A_momentum,BottouL2018R_SGD}. And it is worth pointing out that, in deterministic settings, there is also an acceleration with a convergence rate $\mathcal{O}(1/k^2)$ for the gradient-free methods \cite{NesterovY2017A_GradientFree}.

Many important machine learning models may lead to nonconvex optimization problems \cite{BottouL2018R_SGD,PascanuR2013A_DifficultyRNN}. So we also provide analyses when the objective function is nonconvex, twice differentiable and bounded below. It is shown that all these methods converge to a solution with a zero expected gradient norm. 

The remainder of the paper is organized as follows. The next section introduces the assumptions of objectives and the mean-variance framework. In \cref{SGFD:s3}, we propose the stochastic gradient-free descent methods and analyzes their convergence behavior. On the basis of the conclusions obtained in \cref{SGFD:s3}, we further propose the accelerated methods with the rate of convergence $\mathcal{O}(1/k^2)$ in probability for strongly convex objectives in \cref{SGFD:s4}. And finally, we draw some conclusions in \cref{SGFD:s5}.

\section{Assumptions and analysis framework}
\label{SGFD:s2}

In this section, we shall first state several assumptions of the objectives and then describe the mean-variance framework, which is frequently used in the literature \cite{StichS2013A_ConvexDF,BottouL2018R_SGD} and has potential for analyzing a large collection of stochastic optimization methods of various forms and characteristics.

\subsection{Assumptions of objectives}

First, let us begin with a basic assumption of smoothness of the objective function. Such an assumption is essential for convergence analyses of our methods, as well as most gradient-based methods \cite{BottouL2018R_SGD}. 

\begin{assumption}[Lipschitz-continuous gradients]\label{SGFD:ass:A1}
The objective function $F:\mathbb{R}^d\to\mathbb{R}$ is continuously differentiable and its gradient function $\nabla F:\mathbb{R}^d\to \mathbb{R}^d$ is Lipschitz continuous with Lipschitz constant $0<L<\infty$, i.e.,
\begin{equation*}
  \|\nabla F(x')-\nabla F(x)\|_2\leqslant L\|x'-x\|_2
  ~~\textrm{for all}~~x',x\in\mathbb{R}^d.
\end{equation*}
\end{assumption}

\cref{SGFD:ass:A1} ensures that the gradient of the objective $F$ is bounded and does not change arbitrarily quickly with respect to the parameter vector. As an important consequence of \cref{SGFD:ass:A1} we note that
\begin{equation}\label{SGFD:eq:A1C}
  \left|F(x')-F(x)-\nabla F(x)^\mathrm{T}(x'-x)\right|\leqslant
  \frac{L}{2}\|x'-x\|_2^2~~\textrm{for all}~~x',x\in\mathbb{R}^d.
\end{equation}
This inequality comes from 
\begin{align*}
  \left|F(x')-F(x)-\nabla F(x)^\mathrm{T}(x'\!-\!x)\right|
  \leqslant&\int_0^1\left|\Big(\nabla F(x+t(x'\!-\!x))-\nabla F(x)
  \Big)^\mathrm{T}(x'\!-\!x)\right|\ud t \\
  \leqslant&\int_0^1\|\nabla F(x+t(x'\!-\!x))-\nabla F(x)\|_2
  \|x'\!-\!x\|_2\ud t \\
  \leqslant&L\|x'\!-\!x\|_2^2\int_0^1t\ud t=\frac{L}{2}\|x'\!-\!x\|_2^2.
\end{align*}
Moreover, \cref{SGFD:eq:A1C} is trivial if $F$ is twice continuously differentiable with $\|\nabla^2F(x)\|_2\leqslant L$.

Now we formalize a strong convexity assumption, which is often used to ensure a sublinear convergence for the stochastic gradient methods; and the role of strong sonvexity may be essential for such rates of convergence \cite{NemirovskiA2009A_StochasticProgramming,BottouL2018R_SGD}.

\begin{assumption}[Strong convexity]\label{SGFD:ass:ASC}
The objective function $F:\mathbb{R}^d\to\mathbb{R}$ is strongly convex in that there exists a constant $l>0$ such that
\begin{equation}\label{SGFD:eq:ASC}
  F(x')\geqslant F(x)+\nabla F(x)^\mathrm{T}(x'-x)+
  \frac{l}{2}\|x'-x\|_2^2~~\textrm{for all}~~x',x\in\mathbb{R}^d.
\end{equation}
Hence, $F$ has a unique minimizer, denoted as $x_*\in\mathbb{R}^d$ with $F_*:= F(x_*)$.
\end{assumption}

Notice that for any given $x\in\mathbb{R}^d$, the quadratic model $Q(t)=F(x)+\nabla F(x)^\mathrm{T}(t-x)+\frac{l}{2}\|t-x\|_2^2$ has the unique minimizer $t_*=x-\frac{1}{l}\nabla F(x)$ with $Q_*(t_*)=F(x)-\frac{1}{2l}\|\nabla F(x)\|_2^2$, then together with \cref{SGFD:eq:ASC}, one obtain
\begin{equation*}
  F_*\geqslant F(x)+\nabla F(x)^\mathrm{T}(x_*-x)+\frac{l}{2}\|x_*-x\|_2^2
  \geqslant F(x)-\frac{1}{2l}\|\nabla F(x)\|_2^2,
\end{equation*}
that is, for a given point $x\in\mathbb{R}^d$, the gap between the value of the objective and the minima can be bounded by the squared $\ell_2$-norm of the gradient of the objective:
\begin{equation}\label{SGFD:eq:convexity}
  2l(F(x)-F_*)\leqslant\|\nabla F(x)\|_2^2.
\end{equation}
This inequality is usually referred to as the Polyak-{\L}ojasiewicz inequality which was originally introduced by Polyak \cite{PolyakB1963A_Gradient}. It is a sufficient condition for gradient descent to achieve a linear convergence rate; and it is also a special case of the {\L}ojasiewicz inequality proposed in the same year \cite{LojasiewiczS1963A_PolyakGradient}, which gives an upper bound for the distance of a point to the nearest zero of a given real analytic function.

Under \cref{SGFD:ass:A1,SGFD:ass:ASC}, it is very easy to see that $l\leqslant L$. Furthermore, if $F$ is twice continuously differentiable, then \cref{SGFD:ass:A1,SGFD:ass:ASC} also imply that $0<l\leqslant \|\nabla^2F(x)\|_2\leqslant L<\infty$ for every $x\in\mathbb{R}^d$.

Many important machine learning models may lead to nonconvex optimization problems. Hence we formalize the following assumption so that we can also provide meaningful guarantees in nonconvex settings.

\begin{assumption}[General objectives]\label{SGFD:ass:AG}
The objective function $F:\mathbb{R}^d\to\mathbb{R}$ is twice differentiable and bounded below by a scalar $F_{\textrm{inf}}<\infty$, and particularly, the mapping $\|\nabla F(x)\|_2^2:\mathbb{R}^d\to\mathbb{R}_+$ has Lipschitz-continuous derivatives with Lipschitz constant $L_G>0$.
\end{assumption}

\subsection{Mean-variance framework}

The mean-variance framework can be fully described as a fundamental lemma for any iteration based on random steps, which is a slight generalization of Lemma 4.2 in \cite{BottouL2018R_SGD}. This lemma relies only on \cref{SGFD:ass:A1}.

\begin{lemma}\label{SGFD:lem:AL1}
Under \cref{SGFD:ass:A1}, if for every $k\in\mathbb{N}$, $\theta_k$ is any random vector independent of $x_k$ and $s(x_k,\theta_k)$ is a stochastic step depending on $\theta_k$, then the iteration 
\begin{equation*}
  x_{k+1}=x_k+s(x_k,\theta_k)
\end{equation*}
satisfy the following inequality
\begin{align*}
  \mathbb{E}_{\theta_k}[F(x_{k+1})]\!-\!F(x_k)
  \!\leqslant&\nabla F(x_k)^\mathrm{T}
  \mathbb{E}_{\theta_k}[s(x_k,\!\theta_k)]
  \!+\!\frac{L}{2}\|\mathbb{E}_{\theta_k}[s(x_k,\!\theta_k)]\|_2^2\!
  +\!\frac{L}{2}\mathbb{V}_{\theta_k}[s(x_k,\!\theta_k)],
\end{align*}
where the variance of $s(x_k,\theta_k)$ is defined as
\begin{equation}\label{SGFD:eq:V}
  \mathbb{V}_{\theta_k}[s(x_k,\theta_k)]:= \mathbb{E}_{\theta_k}[\|s(x_k,\theta_k)\|_2^2]- \|\mathbb{E}_{\theta_k}[s(x_k,\theta_k)]\|_2^2.
\end{equation}
\end{lemma}
\begin{proof}
By \cref{SGFD:ass:A1}, the iteration $x_{k+1}=x_k+s(x_k,\theta_k)$ satisfy
\begin{align*}
  F(x_{k+1})-F(x_k)\leqslant&\nabla F(x_k)^\mathrm{T}(x_{k+1}-x_k)
  +\frac{L}{2}\|x_{k+1}-x_k\|_2^2 \\
  \leqslant&\nabla F(x_k)^\mathrm{T}s(x_k,\theta_k)
  +\frac{L}{2}\|s(x_k,\theta_k)\|_2^2.
\end{align*}
Noting that $\theta_k$ is independent of $x_k$ and taking expectations in these inequalities with respect to the distribution of $\theta_k$, we obtain
\begin{equation*}
  \mathbb{E}_{\theta_k}[F(x_{k+1})]-F(x_k)\leqslant\nabla F(x_k)^\mathrm{T}\mathbb{E}_{\theta_k}[s(x_k,\theta_k)]
  +\frac{L}{2}\mathbb{E}_{\theta_k}[\|s(x_k,\theta_k)\|_2^2].
\end{equation*}
Recalling \cref{SGFD:eq:V}, we finally get the desired bound.
\end{proof}

Regardless of the states before $x_k$, the expected decrease in the objective function yielded by the $k$th stochastic step $s(x_k,\theta_k)$ could be bounded above by a quantity involving (i) a positive definite quadratic form in the expectation of $s(x_k,\theta_k)$, say,
\begin{equation}\label{SGFD:eq:Q}
  \nabla F(x_k)^\mathrm{T}\mathbb{E}_{\theta_k}[s(x_k,\theta_k)]
  +\frac{L}{2}\|\mathbb{E}_{\theta_k}[s(x_k,\theta_k)]\|_2^2,
\end{equation}
and (ii) the variance of $s(x_k,\theta_k)$. This lemma shows that, the bound of the expected decrease $\mathbb{E}_{\theta_k}[F(x_{k+1})]-F(x_k)$ can be obtained by analyzing the expectation and variance of the step $s(x_k,\theta_k)$. Hence, it provides us with a basic analysis framework for any iteration based on random steps.

\section{Stochastic gradient-free descent}
\label{SGFD:s3}

The fundamental idea of gradient-free methods is not to evaluate and apply gradients directly but to learn information about gradients indirectly through \emph{stochastic directions} and corresponding output feedbacks of the objective function. In the following, we first describe the gradient-free method and then analyze its behavior of iterations.

\subsection{Methods}

We now define our SGFD method as \cref{SGFD:alg:SGFD}. The random vector $\zeta_k\in\mathbb{R}^d$ here is referred to as stochastic direction. Very similar to stochastic gradient method \cite{BottouL2018R_SGD}, the algorithm also presumes that three computational tools exist: (i) a mechanism for generating a realization of random variables $\xi_k$ and $\zeta_k$ (with $\{\xi_k\}$ or $\{\zeta_k\}$ representing a sequence of jointly independent random variables); (ii) given an iteration number $k\in\mathbb{N}$, a mechanism for computing a scalar stepsize $\alpha_k>0$; and (iii) given an iterate $x_k\in\mathbb{R}^d$ and the realizations of $\xi_k,\zeta_k$ and $\alpha_k$, a mechanism for computing a stochastic step $s(x_k,\xi_k,\alpha_k,\zeta_k)\in\mathbb{R}^d$.

\begin{algorithm}
\caption{Stochastic Gradient-Free Descent (SGFD) Method}
\label{SGFD:alg:SGFD}
\begin{algorithmic}[1]
\STATE{Choose an initial iterate $x_1$.}
\FOR{$k=1,2,\cdots$}
\STATE{Generate a realization of the random variables $\xi_k$ and $\zeta_k$.}
\STATE{Choose a stepsize $\alpha_k>0$.}
\STATE{Compute a gradient-free stochastic step $s(x_k,\xi_k,\alpha_k,\zeta_k)$.}
\STATE{Set the new iterate as $x_{k+1}=x_k+s(x_k,\xi_k,\alpha_k,\zeta_k)$.}
\ENDFOR
\end{algorithmic}
\end{algorithm}

We consider the following four choices of the stochastic step $s(x_k,\xi_k,\alpha_k,\zeta_k)$:
\begin{equation}\label{SGFD:eq:ss}
\begin{aligned}
  s(x_k,\xi_k,\alpha_k,\zeta_k)\!=\!\left\{\!\!\!\!
  \begin{array}{c}
    \Big(f(x_k,\xi_k)-f(x_k+\alpha_k\zeta_k,\xi_k)\Big)\zeta_k, \\[0.75em]
    \frac{1}{n_k}\sum_{i=1}^{n_k}
    \Big(f(x_k,\xi_{k,i})-f(x_k+\alpha_k\zeta_k,\xi_{k,i})\Big)\zeta_k,
    \\[0.75em]
    \frac{1}{n_k}\!\sum_{i=1}^{n_k}\!\!\Big[\!\frac{1}{m_k}\!
    \sum_{j=1}^{m_k}\!\!\Big(f(x_k,\xi_{k,i,j})\!-\!f(x_k\!+\!
    \alpha_k\zeta_{k,i},\xi_{k,i,j})\!\Big)\Big]\zeta_{k,i},\\[0.75em]
    \frac{1}{n_k}\sum_{i=1}^{n_k}\Big(f(x_k,\xi_{k,i})-f(x_k
    +\alpha_k\zeta_{k,i},\xi_{k,i})\Big)\zeta_{k,i},
  \end{array}\right.
\end{aligned}
\end{equation}
where the value of the random variables $\xi_k$ and $\zeta_k$ need only be viewed as a seed for generating a stochastic step, and both $\{\xi_k\}$ and $\{\zeta_k\}$ are independent and identically distributed for every $k\in\mathbb{N}$. These four different possible choices in \cref{SGFD:eq:ss} depend on how we coordinate the random selection of data samples and directions. And it follows from \cref{SGFD:eq:f,SGFD:eq:ss} that
\begin{equation}\label{SGFD:eq:Ess}
  \mathbb{E}_{\xi_k,\zeta_k}[s(x_k,\xi_k,\alpha_k,\zeta_k)]
  =\mathbb{E}_{\zeta_k}
  \left[\Big(F(x_k)-F(x_k+\alpha_k\zeta_k)\Big)\zeta_k\right],
\end{equation}
which is referred to as an \emph{average descent direction} with respect to the stepsize $\alpha_k$ and the distribution of $\zeta_k$. And we use $\mathbb{E}[\cdot]$ to denote an expected value taken with respect to the joint distribution of all random variables, that is, the total expectation operator can be defined as
\begin{equation}\label{SGFD:eq:totalE}
  \mathbb{E}[\cdot]:=\mathbb{E}_{\xi_1,\zeta_1}\mathbb{E}_{\xi_2,\zeta_2} \cdots\mathbb{E}_{\xi_k,\zeta_k}[\cdot].
\end{equation}

Correspondingly, when  the objective function $F$ is able and easy to calculate, we would also consider the following two choices of the stochastic step:
\begin{equation}\label{SGFD:eq:fss}
\begin{aligned}
  s(x_k,\alpha_k,\zeta_k)=\left\{
  \begin{array}{c}
    \Big(F(x_k)-F(x_k+\alpha_k\zeta_k)\Big)\zeta_k, \\[0.75em]
    \frac{1}{n_k}\sum_{i=1}^{n_k}
    \Big(F(x_k)-F(x_k+\alpha_k\zeta_{k,i})\Big)\zeta_{k,i},
  \end{array}\right.
\end{aligned}
\end{equation}
which can be seen as special cases of  \cref{SGFD:eq:ss}.

\begin{lemma}\label{SGFD:lem:Estep}
Under \cref{SGFD:ass:A1}, the expectation of the stochastic steps \cref{SGFD:eq:ss} satisfy that for every $k\in\mathbb{N},\alpha_k>0$ and $x_k\in\mathbb{R}^d$, there is a $C_{x_k,\alpha_k\zeta_k}\in[-L,L]$, which depends on $x_k$ and $\alpha_k\zeta_k$ for every $\zeta_k$, such that
\begin{equation*}
  \mathbb{E}_{\xi_k,\zeta_k}[s(x_k,\xi_k,\alpha_k,\zeta_k)]
  =-\alpha_k\mathbb{E}_{\zeta_k}\!\!
  \left[\Big(\nabla F(x_k)^\mathrm{T}\zeta_k\Big)\zeta_k\right]
  -\frac{\alpha_k^2}{2}\mathbb{E}_{\zeta_k}\!\!
  \left[C_{x_k,\alpha_k\zeta_k}
  \Big(\zeta_k^\mathrm{T}\zeta_k\Big)\zeta_k\right].
\end{equation*}
\end{lemma}
\begin{proof}
According to \cref{SGFD:eq:A1C}, there is a $C_{x_k,\alpha_k\zeta_k} \in[-L,L]$ for every $\zeta_k$ such that
\begin{align*}
  \Big(F(x_k)-F(x_k+\alpha_k\zeta_k)\Big)\zeta_k
  =-\alpha_k\Big(\nabla F(x_k)^\mathrm{T}\zeta_k\Big)\zeta_k
  -\frac{\alpha_k^2}{2}C_{x_k,\alpha_k\zeta_k}
  \Big(\zeta_k^\mathrm{T}\zeta_k\Big)\zeta_k,
\end{align*}
taking expectations with respect to the distribution of $\zeta_k$ and recalling \cref{SGFD:eq:Ess}, one obtains
\begin{align*}
  \mathbb{E}_{\xi_k,\zeta_k}[s(x_k,\xi_k,\alpha_k,\zeta_k)]
  &=\mathbb{E}_{\zeta_k}\!\!
  \left[\Big(F(x_k)-F(x_k+\alpha_k\zeta_k)\Big)\zeta_k\right] \\
  &=-\alpha_k\mathbb{E}_{\zeta_k}\!\!
  \left[\Big(\nabla F(x_k)^\mathrm{T}\zeta_k\Big)\zeta_k\right]
  -\frac{\alpha_k^2}{2}\mathbb{E}_{\zeta_k}\!\!
  \left[C_{x_k,\alpha_k\zeta_k}
  \Big(\zeta_k^\mathrm{T}\zeta_k\Big)\zeta_k\right],
\end{align*}
as desired.
\end{proof}

\subsection{Distribution of random directions}

We now formalize an assumption of distribution of random directions as follows. 
\begin{assumption}\label{SGFD:ass:A2}
The $d$-dimensional random vectors $\{\zeta_k\}$ are independent and identically distributed and simultaneously, for every $k\in\mathbb{N}$, $\zeta_k\in\mathbb{R}^d$ satisfy (i) the mean of each component is $0$, i.e.,
\begin{equation}\label{SGFD:eq:A2E}
  \mathbb{E}\big[\zeta_k^{(i)}\big]=0
  ~~\textrm{for all}~~i\in\{1,\cdots,d\};
\end{equation}
(ii) the covariance matrix is a unit matrix, i.e.,
\begin{equation}\label{SGFD:eq:A2C}
  \mathbb{E}\big[\zeta_k^{(i)}\zeta_k^{(j)}\big]=\left\{
  \begin{array}{cc}
    1, & i=j, \\
    0, & i\neq j;
  \end{array}\right.
\end{equation}
and (iii) every component is bounded or has a finite nonzero fourth moment, i.e., for all $i\in\{1,\cdots,d\}$, it holds that
\begin{equation}\label{SGFD:eq:A2F}
  \big|\zeta_k^{(i)}\big|\leqslant r_\zeta~~~\textrm{or}~~~
  \mathbb{E}\big[\zeta_k^{(i)}\big]^4\leqslant m_\zeta^{(4)},
\end{equation}
where $\zeta_k^{(i)}$ is the $i$th element in vector $\zeta_k$.
\end{assumption}

One of the typical choices for the distribution of $\zeta_k$ is, of course, the $d$-dimensional standard normal distribution with zero mean and unit covariance matrix, whose each component is obviously unbounded but has a finite fourth moment, say, $\mathbb{E}\big[\zeta_k^{(i)}\big]^4=3$ for all $i\in\{1,\cdots,d\}$; another typical choice is the $d$-dimensional uniform distribution on $[-\sqrt{3},\sqrt{3}]^d$, whose each component is bounded and has a finite fourth moment.

Under \cref{SGFD:ass:A2} alone, we could obtain the following lemma. Such a result is essential for convergence analyses of all our methods. And in fact, although not so intuitive, this lemma is a direct source of the basic idea for this work.
\begin{lemma}\label{SGFD:lem:E13}
Under \cref{SGFD:ass:A2}, for every vector $\omega\in\mathbb{R}^d$ independent of $\zeta_k$, it follows that
\begin{equation}\label{SGFD:eq:E}
  \omega=\mathbb{E}_{\zeta_k}\!\!
  \left[\big(\omega^\mathrm{T}\zeta_k\big)\zeta_k\right],
\end{equation}
and
\begin{equation}\label{SGFD:eq:E3}
  \left\|\mathbb{E}_{\zeta_k}\big[\big(\zeta_k^\mathrm{T} \zeta_k\big)
  \zeta_k\big]\right\|_\infty\leqslant dD_\zeta~~\textrm{and}~~
  \left\|\mathbb{E}_{\zeta_k}\big[\big(\zeta_k^\mathrm{T} \zeta_k\big)
  \zeta_k\big]\right\|_2\leqslant d^{\frac{3}{2}}D_\zeta,
\end{equation}
where $\{\zeta_k\}$ are independent and identically distributed,
\begin{equation*}
  D_\zeta=\min\Big(r_\zeta,\sqrt{1+m_\zeta^{(4)}/d-1/d}\Big),
\end{equation*}
$r_\zeta=\|\zeta_k\|_\infty$, and $m_\zeta^{(4)}$ is the fourth moment of $\zeta_k$.
\end{lemma}
\begin{proof}
Let $v=\mathbb{E}_{\zeta_k}[(\omega^\mathrm{T}\zeta_k)\zeta_k]$, then
\begin{equation*}
  v=\mathbb{E}_{\zeta_k}\!\bigg[\bigg(
  \sum_{i=1}^d\omega^{(i)}\zeta_k^{(i)}\bigg)\zeta_k\bigg]
  =\sum_{i=1}^d\omega^{(i)}\mathbb{E}_{\zeta_k}\!
  \big[\zeta_k^{(i)}\zeta_k\big],
\end{equation*}
from \cref{SGFD:ass:A2} we further observed that for all $1\leqslant j\leqslant d$,
\begin{equation*}
  v^{(j)}=\sum_{i=1}^d\omega^{(i)}\mathbb{E}_{\zeta_k}
  \big[\zeta_k^{(i)}\zeta_k^{(j)}\big]=\omega^{(j)},~~\textrm{that is},~~
  \omega=\mathbb{E}_{\zeta_k}[(\omega^\mathrm{T}\zeta_k)\zeta_k].
\end{equation*}

Secondly, from \cref{SGFD:ass:A2}, if every component of $\zeta_k$ is bounded, i.e., $|\zeta_k^{(i)}|\leqslant r_\zeta$, then one obtains
\begin{equation}\label{SGFD:eq:E31}
  \left\|\mathbb{E}_{\zeta_k}\big[\big(\zeta_k^\mathrm{T} \zeta_k\big)
  \zeta_k\big]\right\|_\infty
  \leqslant r_\zeta\bigg|\sum_{i=1}^d\mathbb{E}_{\zeta_k}
  \left[\zeta_k^{(i)}\zeta_k^{(i)}\right]\bigg|\leqslant dr_\zeta;
\end{equation}
or, if every component of $\zeta_k$ has a finite nonzero fourth moment $m_\zeta^{(4)}$, then according to Cauchy-Schwartz's inequality, it follows that
\begin{align*}
  \left\|\mathbb{E}_{\zeta_k}\big[\big(\zeta_k^\mathrm{T} \zeta_k\big)
  \zeta_k\big]\right\|_\infty\leqslant&\max_{1\leqslant j\leqslant d}
  \mathbb{E}_{\zeta_k}\bigg|\bigg(
  \sum_{i=1}^d\zeta_k^{(i)}\zeta_k^{(i)}\bigg)\zeta_k^{(j)}\bigg| \\
  \leqslant&\max_{1\leqslant j\leqslant d}\left[\mathbb{E}_{\zeta_k}
  \left(\zeta_k^{(j)}\zeta_k^{(j)}\right)\right]^{\frac{1}{2}}
  \bigg[\mathbb{E}_{\zeta_k}\bigg(\sum_{i=1}^d\zeta_k^{(i)}
  \zeta_k^{(i)}\bigg)^2\bigg]^{\frac{1}{2}}
\end{align*}
together with \cref{SGFD:eq:A2C}, we further have 
\begin{equation}\label{SGFD:eq:E32}
\begin{split}
  \left\|\mathbb{E}_{\zeta_k}\big[\big(\zeta_k^\mathrm{T}\zeta_k\big)
  \zeta_k\big]\right\|_\infty\leqslant&\bigg[\mathbb{E}_{\zeta_k}
  \bigg(\sum_{1\leqslant i,s\leqslant d}\zeta_k^{(i)}\zeta_k^{(i)}
  \zeta_k^{(s)}\zeta_k^{(s)}\bigg)\bigg]^{\frac{1}{2}} \\
  \leqslant&\bigg[\sum_{i=1}^d
  \mathbb{E}_{\zeta_k}\big(\zeta_k^{(i)}\big)^4
  +\sum_{i\neq l}\mathbb{E}_{\zeta_k}\big(
  \zeta_k^{(i)}\zeta_k^{(i)}\big)\mathbb{E}_{\zeta_k}\big(
  \zeta_k^{(l)}\zeta_k^{(l)}\big)\bigg]^{\frac{1}{2}} \\
  \leqslant&\sqrt{d\big(d+m_\zeta^{(4)}-1\big)},
\end{split}
\end{equation}
then it follows from \cref{SGFD:eq:E31,SGFD:eq:E32} that $\|\mathbb{E}_{\zeta_k}[(\zeta_k^\mathrm{T}\zeta_k)\zeta_k]\|_\infty\leqslant dD_\zeta$; and further, 
\begin{equation*}
  \|\mathbb{E}_{\zeta_k} [(\zeta_k^\mathrm{T}\zeta_k)\zeta_k]\|_2
  \leqslant d^{\frac{1}{2}}\|\mathbb{E}_{\zeta_k} [(\zeta_k^\mathrm{T}\zeta_k)\zeta_k]\|_\infty
  \leqslant d^{\frac{3}{2}}D_\zeta,
\end{equation*}
and the proof is complete.
\end{proof}

\subsection{Expectation analysis}

The important role of \cref{SGFD:ass:A2} is to ensure that $\frac{1}{\alpha_k}\mathbb{E}_{\xi_k,\zeta_k} [s(x_k,\xi_k,\alpha_k,\zeta_k)]$ an asymptotic unbiased estimator of $-\nabla F(x_k)$. And it also allows us to analyze the bound of the quadratic form \cref{SGFD:eq:Q} in the expectation of stochastic directions.

\begin{theorem}\label{SGFD:thm:step}
Under \cref{SGFD:ass:A1,SGFD:ass:A2}, the stochastic steps of SGFD (\cref{SGFD:alg:SGFD}) satisfy the asymptotic unbiasedness
\begin{equation}\label{SGFD:eq:limit}
  \lim_{\alpha_k\to0}\frac{\mathbb{E}_{\xi_k,\zeta_k}
  \big[s(x_k,\xi_k,\alpha_k,\zeta_k)\big]}{\alpha_k}=-\nabla F(x_k),
\end{equation}
and for every $k\in\mathbb{N}$, it follows that
\begin{align}
  \|\mathbb{E}_{\xi_k,\zeta_k}[s(x_k,\xi_k,\alpha_k,\zeta_k)]\|_2
  \leqslant&\alpha_k\|\nabla F(x_k)\|_2
  +\frac{\alpha_k^2Ld^{\frac{3}{2}}D_\zeta}{2}
  ~~\textrm{and}~~\label{SGFD:eq:step1} \\
  \nabla F(x_k)^\mathrm{T}\mathbb{E}_{\xi_k,\zeta_k}
  [s(x_k,\xi_k,\alpha_k,\zeta_k)]\leqslant&
  -\alpha_k\|\nabla F(x_k)\|_2^2
  +\frac{\alpha_k^2Ld^{\frac{3}{2}}D_\zeta}{2}\|\nabla F(x_k)\|_2,
  \label{SGFD:eq:step2}
\end{align}
where the constant $D_\zeta$ comes from \cref{SGFD:lem:E13}.
\end{theorem}
\begin{proof}
According to \cref{SGFD:lem:Estep}, there are $C_{x_k,\alpha_k\zeta_k}\in[-L,L]$ such that
\begin{equation*}
  \mathbb{E}_{\xi_k,\zeta_k}[s(x_k,\xi_k,\alpha_k,\zeta_k)]
  =-\alpha_k\mathbb{E}_{\zeta_k}\!\!
  \left[\Big(\nabla F(x_k)^\mathrm{T}\zeta_k\Big)\zeta_k\right]
  -\frac{\alpha_k^2}{2}\mathbb{E}_{\zeta_k}\!\!
  \left[C_{x_k,\alpha_k\zeta_k}
  \Big(\zeta_k^\mathrm{T}\zeta_k\Big)\zeta_k\right].
\end{equation*}
First, it follows from \cref{SGFD:lem:E13} that
\begin{equation*}
  \nabla F(x_k)=\mathbb{E}_{\zeta_k}\!\!
  \left[\Big(\nabla F(x_k)^\mathrm{T}\zeta_k\Big)\zeta_k\right];
\end{equation*}
and then, let $R_k=\mathbb{E}_{\zeta_k}\big[C\big(\zeta_k^\mathrm{T} \zeta_k\big)\zeta_k\big]$, then from \cref{SGFD:ass:A2,SGFD:lem:E13}, there is a $D_\zeta\in(0,\infty)$ such that
\begin{equation*}
  \|R_k\|_\infty\leqslant L\mathbb{E}_{\zeta_k}\big\|\big(\zeta_k^\mathrm{T} \zeta_k\big)\zeta_k\big\|_\infty\leqslant dLD_\zeta~~\textrm{and}~~
  \|R_k\|_2\leqslant L\mathbb{E}_{\zeta_k}\big\|\big(\zeta_k^\mathrm{T} \zeta_k\big)\zeta_k\big\|_2\leqslant d^{\frac{3}{2}}LD_\zeta.
\end{equation*}
Thus, one obtains \cref{SGFD:eq:limit,SGFD:eq:step1} by noting that
\begin{align*}
  &\bigg\|\frac{\mathbb{E}_{\xi_k,\zeta_k}
  \big[s(x_k,\xi_k,\alpha_k,\zeta_k)\big]}{\alpha_k}
  +\nabla F(x_k)\bigg\|_\infty
  \leqslant\frac{\alpha_k}{2}\|R_k\|_\infty~~\textrm{and}~~\\
  &~~\|\mathbb{E}_{\xi_k,\zeta_k}[s(x_k,\xi_k,\alpha_k,\zeta_k)]\|_2
  \leqslant\alpha_k\|\nabla F(x_k)\|_2+\frac{\alpha_k^2}{2}\|R_k\|_2;
\end{align*}
and further obtains \cref{SGFD:eq:step2} by noting that
\begin{align*}
  \nabla F(x_k)^\mathrm{T}\mathbb{E}_{\xi_k,\zeta_k}
  [s(x_k,\xi_k,\alpha_k,\zeta_k)]=&-\alpha_k\|\nabla F(x_k)\|_2^2
  -\frac{\alpha_k^2}{2}\nabla F(x_k)^\mathrm{T}R_k \\
  \leqslant&-\alpha_k\|\nabla F(x_k)\|_2^2
  +\frac{\alpha_k^2}{2}\|\nabla F(x_k)\|_2\|R_k\|_2,
\end{align*}
and the proof is complete.
\end{proof}

\subsection{Variance assumption}

We follow \cite{BottouL2018R_SGD} to make the following assumption about the variance of random steps $s(x_k,\xi_k,\alpha_k,\zeta_k)$. It states that the variance of $s(x_k,\xi_k,\alpha_k,\zeta_k)$ is restricted, but in a relatively minor manner.
\begin{assumption}\label{SGFD:ass:A3}
The objective function and SGFD (\cref{SGFD:alg:SGFD}) satisfy for all $k\in\mathbb{N}$, there exist scalars $M\geqslant0$ and $M_V\geqslant0$ such that
\begin{equation*}
  \mathbb{V}_{\xi_k,\zeta_k}[s(x_k,\xi_k,\alpha_k,\zeta_k)]
  \leqslant\alpha_k^2M+\alpha_k^2M_V\|\nabla F(x_k)\|_2^2.
\end{equation*}
\end{assumption}

According to \cref{SGFD:lem:AL1,SGFD:thm:step,SGFD:ass:A3}, it can be seen that for all $k\in\mathbb{N}$, the difference $\mathbb{E}_{\xi_k,\zeta_k} [F(x_{k+1})]-F(x_k)$ is bounded above by a deterministic quantity. This is what the following lemma states.
\begin{lemma}\label{SGFD:lem:EB}
Under \cref{SGFD:ass:A1,SGFD:ass:A2,SGFD:ass:A3}, now suppose that the SGFD method (\cref{SGFD:alg:SGFD}) is run with $0<\alpha_k\leqslant\frac{1}{L}$ for every $k\in\mathbb{N}$, then
\begin{equation*}
  \mathbb{E}_{\xi_k,\zeta_k}[F(x_{k+1})]-F(x_k)\leqslant
  -\alpha_k\|\nabla F(x_k)\|_2^2
  +\frac{\alpha_k^2LM_G}{2}\|\nabla F(x_k)\|_2^2
  +\frac{\alpha_k^2LM_d}{2},
\end{equation*}
where $M_G=M_V+2$, $M_d=M+\frac{5d^3D_\zeta^2}{4}$, and the constant $D_\zeta$ comes from \cref{SGFD:lem:E13}.
\end{lemma}
\begin{proof}
It follows from \cref{SGFD:eq:step1} that 
\begin{align*}
  \left\|\mathbb{E}_{\xi_k,\zeta_k}
  \big[s(x_k,\xi_k,\alpha_k,\zeta_k)\big]\right\|_2^2
  \leqslant&\bigg(\alpha_k\|\nabla F(x_k)\|_2
  +\frac{\alpha_k^2Ld^{\frac{3}{2}}D_\zeta}{2}\bigg)^2 \\
  \leqslant&\alpha_k^2\|\nabla F(x_k)\|_2^2
  +\alpha_k^3Ld^{\frac{3}{2}}D_\zeta\|\nabla F(x_k)\|_2
  +\frac{\alpha_k^4L^2d^3D_\zeta^2}{4},
\end{align*}
together with the Arithmetic Mean Geometric Mean inequality, i.e.,
\begin{equation*}
  \alpha_k^3Ld^{\frac{3}{2}}D_\zeta\|\nabla F(x_k)\|_2
  =\alpha_k\|\nabla F(x_k)\|_2\cdot\alpha_k^2Ld^{\frac{3}{2}}D_\zeta
  \leqslant\frac{\alpha_k^2}{2}\|\nabla F(x_k)\|_2^2
  +\frac{\alpha_k^4L^2d^3D_\zeta^2}{2},
\end{equation*}
thus, by noting that $\alpha_kL<1$, one obtains
\begin{equation*}
  \left\|\mathbb{E}_{\xi_k,\zeta_k}
  \big[s(x_k,\xi_k,\alpha_k,\zeta_k)\big]\right\|_2^2
  \leqslant\frac{3\alpha_k^2}{2}\|\nabla F(x_k)\|_2^2
  +\frac{3\alpha_k^2d^3D_\zeta^2}{4}.
\end{equation*}
Similarly, by \cref{SGFD:eq:step2} and the Arithmetic Mean Geometric Mean inequality, it holds that
\begin{align*}
  \nabla F(x_k)^\mathrm{T}\mathbb{E}_{\xi_k,\zeta_k}
  \big[s(x_k,\xi_k,\alpha_k,\zeta_k)\big]
  \leqslant&-\alpha_k\|\nabla F(x_k)\|_2^2
  +\frac{\alpha_k^2L}{2}d^{\frac{3}{2}}D_\zeta\|\nabla F(x_k)\|_2 \\
  \leqslant&-\alpha_k\|\nabla F(x_k)\|_2^2
  +\frac{\alpha_k^2L}{4}\|\nabla F(x_k)\|_2^2
  +\frac{\alpha_k^2Ld^3D_\zeta^2}{4}.
\end{align*}
Finally, by \cref{SGFD:lem:AL1} and \cref{SGFD:ass:A3}, the iterates generated by SGFD satisfy
\begin{align*}
  \mathbb{E}_{\xi_k,\zeta_k}[F(x_{k+1})]\!-\!F(x_k)
  \leqslant&\nabla F(x_k)^\mathrm{T}\mathbb{E}_{\xi_k,\zeta_k}
  \big[s(x_k,\xi_k,\alpha_k,\zeta_k)\big] \\
  &+\!\frac{L}{2}\|\mathbb{E}_{\xi_k,\zeta_k}\!
  \big[s(x_k,\xi_k,\alpha_k,\zeta_k)\big]\|_2^2\!
  +\!\frac{L}{2}\mathbb{V}_{\xi_k,\zeta_k}\!
  \big[s(x_k,\xi_k,\alpha_k,\zeta_k)\big] \\
  \leqslant&-\alpha_k\|\nabla F(x_k)\|_2^2
  +\frac{\alpha_k^2LM_G}{2}\|\nabla F(x_k)\|_2^2
  +\frac{\alpha_k^2LM_d}{2},
\end{align*}
and the proof is complete.
\end{proof}

This lemma reveals that regardless of how the method arrived at the iterate $x_k$, the optimization process continues in a Markovian manner in the sense that $x_{k+1}$ is a random variable that depends only on the iterate $x_k$, the seeds $\xi_k$ and $\zeta_k$, and the stepsize $\alpha_k$ and not on any past iterates. This is the same as in the case of stochastic gradients \cite{BottouL2018R_SGD}.

\subsection{Average behavior of iterations}

Here we mainly focus on the strongly convex objective functions. According to \cref{SGFD:lem:EB}, it is easy to analyze the average behavior of the SGFD iterations for strongly convex objectives.
\begin{theorem}\label{SGFD:thm:AB}
Under \cref{SGFD:ass:A1,SGFD:ass:A2,SGFD:ass:A3,SGFD:ass:ASC}, now suppose that the SGFD method (\cref{SGFD:alg:SGFD}) is run with $0<\alpha_k\leqslant\frac{1}{LM_G}$ for every $k\in\mathbb{N}$, then
\begin{equation*}
  \mathbb{E}[F(x_{k+1})-F_*]\leqslant
  \bigg(\prod_{i=1}^k(1-\alpha_il)\bigg)[F(x_1)-F_*]
  +\frac{LM_d}{2}\sum_{i=1}^k\alpha_i^2\prod_{j=i+1}^k(1-\alpha_jl).
\end{equation*}
\end{theorem}
\begin{proof}
According to \cref{SGFD:lem:EB} and $0<\alpha_k\leqslant\frac{1}{LM_G}$, we have
\begin{align*}
  \mathbb{E}_{\xi_k,\zeta_k}[F(x_{k+1})]\leqslant&
  F(x_k)-\alpha_k\|\nabla F(x_k)\|_2^2
  +\frac{\alpha_k^2LM_G}{2}\|\nabla F(x_k)\|_2^2
  +\frac{\alpha_k^2LM_d}{2} \\
  \leqslant&F(x_k)-\frac{\alpha_k}{2}\|\nabla F(x_k)\|_2^2
  +\frac{\alpha_k^2LM_d}{2}.
\end{align*}
Subtracting $F_*$ from both sides and applying \cref{SGFD:eq:convexity}, this yields
\begin{align*}
  \mathbb{E}_{\xi_k,\zeta_k}[F(x_{k+1})]-F_*
  \leqslant&F(x_k)-F_*-\frac{\alpha_k}{2}\|\nabla F(x_k)\|_2^2
  +\frac{\alpha_k^2LM_d}{2} \\
  \leqslant&F(x_k)-F_*-\alpha_kl\big(F(x_k)-F_*\big)
  +\frac{\alpha_k^2LM_d}{2} \\
  =&(1-\alpha_kl)\big(F(x_k)-F_*\big)+\frac{\alpha_k^2LM_d}{2},
\end{align*}
and it follows from taking total expectations that
\begin{equation*}
  \mathbb{E}[F(x_{k+1})-F_*]\leqslant(1-\alpha_kl)
  \mathbb{E}[F(x_k)-F_*]+\frac{\alpha_k^2LM_d}{2}.
\end{equation*}
Thus, the desired result follows by repeatedly applying this inequality above through iteration from $1$ to $k$.
\end{proof}

\subsection{Convergence for strongly convex objectives}

Here we will use a new analysis approach based on the properties of the gamma function. First, according to \cref{SGFD:thm:AB}, under strong convex conditions, the convergence of the SGFD methods is closely related to the following two limits:
\begin{equation}\label{SGFD:eq:limits}
  A:=\lim_{k\to\infty}A_k~~~\textrm{and}~~~B:=\lim_{k\to\infty}B_k,
\end{equation}
where $A_k=\prod_{i=1}^k(1-\alpha_il)$ and $B_k=\sum_{i=1}^k\alpha_i^2\prod_{j=i+1}^k(1-\alpha_jl)$. Therefore, the result in \cref{SGFD:thm:AB} can be rewritten as
\begin{equation*}
  \mathbb{E}[F(x_{k+1})-F_*]\leqslant[F(x_1)-F_*]A_k
  +\frac{LM_d}{2}B_k.
\end{equation*}
More specifically, if there is a non-increasing stepsize sequence $\{\alpha_k\}$, satisfying $\alpha_k\leqslant\frac{1}{LM_G}$, such that $A=0$ and $B=0$, then we shall obtain the expected convergence $\lim_{k\to\infty}\mathbb{E}[F(x_k)]=F_*$. 

Now let us analyze the requirements that $A$ and $B$ are equal to $0$ at the same time. First, since $\alpha_kl<1$ for every $k\in\mathbb{N}$, we have
\begin{equation*}
  B=\lim_{k\to\infty}\sum_{i=1}^k\alpha_i^2\!\prod_{j=i+1}^k\!(1-\alpha_jl)
  >\!\bigg(\lim_{k\to\infty}\sum_{i=1}^k\alpha_i^2\bigg)
  \!\bigg(\lim_{k\to\infty}\prod_{i=1}^k(1-\alpha_il)\!\bigg)\!
  =A\lim_{k\to\infty}\sum_{i=1}^k\alpha_i^2,
\end{equation*}
therefore, $B=0$ implies $A=0$. This requires that the stepsize sequence $\{\alpha_k\}$ cannot decay too quickly; for example, if we choose a stepsize sequence $\alpha_k=\frac{\beta}{k^2}$ with $\beta<l^{-1}$ for every $k\in\mathbb{N}$, then according to Euler's product formula for sine, we have
\begin{equation*}
  A=\lim_{k\to\infty}\prod_{i=1}^k\left(1-\frac{\beta l}{i^2}\right)
  =\frac{\sin(\sqrt{\beta l}\pi)}{\sqrt{\beta l}\pi}\neq0,
  ~~\textrm{where}~~0<\beta l<1.
\end{equation*}
And otherwise, if we choose a fixed stepsize $\alpha_k=\bar{\alpha}=\frac{\bar{\alpha}}{k^0}$ for every $k\in\mathbb{N}$, then we will find out
\begin{equation*}
  B=\bar{\alpha}^2\lim_{k\to\infty}\sum_{i=1}^k(1-\bar{\alpha}l)^{k-i}
  =\frac{\bar{\alpha}}{l}\neq0;
\end{equation*}
and by further noting that $A=\lim_{k\to\infty}(1-\bar{\alpha}l)^k=0$ at the same time, so it follows that, similar to the case of the stochastic gradient descent method \cite{BottouL2018R_SGD}, the expected objective values will converge linearly to a neighborhood of the optimal value, but the random directions prevent further progress after some point; and it is apparent that selecting a smaller stepsize worsens the contraction constant $1-\bar{\alpha}l$ in the convergence rate, but allows one to arrive closer, i.e., bounded by $\frac{\bar{\alpha}LM_d}{2l}$, to the optimal value.

This implies that when the stepsize $\alpha_k$ has a reasonable decay with respect to $k$, the expected convergence can be guaranteed with $A=B=0$. So, if we assume that the stepsize $\alpha_k$ decays like $k^{-s}$ for $k\to\infty$, then $s$ must satisfy $0<s<2$, otherwise the limits $A$ and $B$ cannot be equal to $0$ at the same time.

Let us consider a stepsize sequence such that, for every $k\in\mathbb{N}$, \begin{equation}\label{SGFD:eq:stepsize}
  \alpha_k=\frac{\beta}{k+\sigma}~~\textrm{for some}~~\beta>\frac{1}{l}~~
  \textrm{and}~~\sigma>0~~\textrm{such that}~~\alpha_1\leqslant\frac{1}{LM_G}.
\end{equation}
This is a special case of the stepsize requirement given in the seminal work of Robbins and Monro \cite{RobbinsH1951A_SG}, which takes the form
\begin{equation*}
  \sum_{k=1}^\infty\alpha_k=\infty~~\textrm{and}~~
  \sum_{k=1}^\infty\alpha_k^2<\infty.
\end{equation*}

If we take a stepsize sequence satisfying \cref{SGFD:eq:stepsize}, then
\begin{equation}\label{SGFD:eq:Ak}
  A_k=\prod_{i=1}^k\frac{i+\sigma-\beta l}{i+\sigma}
  =\frac{(1+\sigma-\beta l)_k}{(1+\sigma)_k}
  =\frac{\Gamma(1+\sigma)}{\Gamma(1+\sigma-\beta l)}
  \frac{\Gamma(k+1+\sigma-\beta l)}{\Gamma(k+1+\sigma)},
\end{equation}
where the shifted factorial or Pochhammer symbol $(z)_k$ is $(z)_k=z(z+1)\cdots(z+k-1)$ and $\Gamma(z)$ is the gamma function for all $z\neq0,-1,-2,\cdots$; in the last identity above, we applied the relationship
\begin{equation*}
  (z)_k\Gamma(z)=\Gamma(k+z),
\end{equation*}
which is extended from the recursive formula $z\Gamma(z)=\Gamma(1+z)$. The following lemma gives the first-order asymptotic expansion of the ratio of two gamma functions:
\begin{lemma}[Tricomi and Erd\'{e}lyi in \cite{TricomiF1951A_GammaRatio}]
\label{SGFD:lem:GammaRatio}
For any $a\in\mathbb{R}$,
\begin{equation*}
  \frac{\Gamma(x+a)}{\Gamma(x)}=x^a+\bm{O}\left(x^{a-1}\right)
\end{equation*}
as $x\to\infty$.
\end{lemma}

\begin{lemma}
If a stepsize sequence takes the form \cref{SGFD:eq:stepsize}, then we have the following first-order asymptotic expansions
\begin{equation*}
  A_k=\frac{\Gamma(1+\sigma)}{\Gamma(1+\sigma-\beta l)}
  (k+1+\sigma)^{-\beta l}+\bm{O}\left((k+1+\sigma)^{-1-\beta l}\right)
\end{equation*}
and
\begin{equation*}
  B_k=\frac{C_k\beta^2}{\beta l-1}(k+1+\sigma)^{-1}
  +\bm{O}\left((k+1+\sigma)^{-2}\right)
\end{equation*}
as $k\to\infty$, where $C_k$ tends to $1$ as $k\to\infty$.
\end{lemma}
\begin{proof}
First, the first-order asymptotic expansion of $A_k$ can be directly obtained from \cref{SGFD:eq:Ak,SGFD:lem:GammaRatio}. And similar to \cref{SGFD:eq:Ak}, we obtain
\begin{equation*}
  B_k=\sum_{i=1}^k\frac{\beta^2}{(i\!+\!\sigma)^2}
  \!\prod_{j=i+1}^k\!\frac{j\!+\!\sigma\!-\!\beta l}{j\!+\!\sigma}
  =\beta^2\frac{\Gamma(k\!+\!1\!+\!\sigma\!-\!\beta l)}
  {\Gamma(k\!+\!1\!+\!\sigma)}\sum_{i=1}^k\frac{1}{(i\!+\!\sigma)^2}
  \frac{\Gamma(i\!+\!1\!+\!\sigma)}{\Gamma(i\!+\!1\!+\!\sigma\!-\!\beta l)}.
\end{equation*}
And it follows from \cref{SGFD:lem:GammaRatio} that
\begin{align*}
  \sum_{i=1}^k\frac{1}{(i\!+\!\sigma)^2}
  \frac{\Gamma(i\!+\!1\!+\!\sigma)}{\Gamma(i\!+\!1\!+\!\sigma\!-\!\beta l)}
  =&\sum_{i=1}^k\frac{(i\!+\!1\!+\!\sigma\!-\!\beta l)^{\beta l}}
  {(i\!+\!\sigma)^2}+\bm{O}\left(\sum_{i=1}^k
  \frac{(i\!+\!1\!+\!\sigma\!-\!\beta l)^{\beta l-1}}
  {(i+\sigma)^2}\right) \\
  =&\frac{C_k}{\beta l-1}(k+1+\sigma)^{\beta l-1}
  +\bm{O}\left((k+1+\sigma)^{\beta l-2}\right),
\end{align*}
where $C_k$ tends to $1$ as $k\to\infty$; in the last identity above, we applied the relationship $\bm{O}\big(\sum_{i=1}^ki^a\big)=\int_0^k(t+1)^a\ud t=\frac{(k+1)^{a+1}-1}{a+1}$ for any $a\in\mathbb{R}$. Hence, we finally get
\begin{align*}
  B_k=\frac{C_k\beta^2}{\beta l-1}(k+1+\sigma)^{-1}
  +\bm{O}\left((k+1+\sigma)^{-2}\right),
\end{align*}
as desired.
\end{proof}

\begin{theorem}\label{SGFD:thm:main1}
Under \cref{SGFD:ass:A1,SGFD:ass:A2,SGFD:ass:A3,SGFD:ass:ASC}, suppose that the SGFD method (\cref{SGFD:alg:SGFD}) is run with a stepsize sequence $\{\alpha_k\}$ taking the form \cref{SGFD:eq:stepsize}, then
\begin{equation*}
  \mathbb{E}[F(x_k)]-F_*\leqslant
  \frac{C_A\Gamma(1+\sigma)}{\Gamma(1+\sigma-\beta l)}
  \frac{F(x_1)-F_*}{(k+1+\sigma)^{\beta l}}
  +\frac{C_B\beta^2}{2(\beta l-1)}\frac{LM_d}{k+1+\sigma}.
\end{equation*}
According to \cref{SGFD:eq:stepsize}, we have $\beta l>1$, then it holds that $\mathbb{E}[F(x_k)]-F_*=\mathcal{O}(1/k)$.
\end{theorem}

\subsection{Convergence for general objectives}

Similar to the case of the classical stochastic gradient method, although the nonconvex objective functions might possess multiple local minima and other stationary points, we can also provide meaningful guarantees for the SGFD method in nonconvex settings.

\begin{lemma}\label{SGFD:lem:nonconvex}
Under \cref{SGFD:ass:A1,SGFD:ass:A2,SGFD:ass:A3,SGFD:ass:AG}, suppose that the sequence of iterates $\{x_k\}$ is generated by the SGFD method (\cref{SGFD:alg:SGFD}) with a stepsize sequence $\{\alpha_k\}$ taking the form \cref{SGFD:eq:stepsize} and satisfying $\alpha_1\leqslant\frac{1}{LM_G}$, then
\begin{equation*}
  \sum_{i=1}^\infty\alpha_i\mathbb{E}\big[\|\nabla F(x_k)\|_2^2\big]<\infty,
\end{equation*}
and therefore, $\liminf_{k\to\infty}\mathbb{E}[\|\nabla F(x_k)\|_2^2]=0$.
\end{lemma}
\begin{proof}
According to \cref{SGFD:lem:EB} and $0<\alpha_k\leqslant\frac{1}{LM_G}$, we have
\begin{align*}
  \mathbb{E}_{\xi_k,\zeta_k}[F(x_{k+1})]-F(x_k)\leqslant&
  -\alpha_k\|\nabla F(x_k)\|_2^2
  +\frac{\alpha_k^2LM_G}{2}\|\nabla F(x_k)\|_2^2
  +\frac{\alpha_k^2LM_d}{2} \\
  \leqslant&-\frac{\alpha_k}{2}\|\nabla F(x_k)\|_2^2
  +\frac{\alpha_k^2LM_d}{2},
\end{align*}
and it follows from taking total expectations that
\begin{equation*}
  \mathbb{E}[F(x_{k+1})]-\mathbb{E}[F(x_k)]\leqslant-\frac{\alpha_k}{2}
  \mathbb{E}\big[\|\nabla F(x_k)\|_2^2\big]+\frac{\alpha_k^2LM_d}{2},
\end{equation*}
by repeatedly applying this inequality above through iteration from $1$ to $k$, we get
\begin{equation*}
  F_{\textrm{inf}}-F(x_1)\leqslant\mathbb{E}[F(x_{k+1})]-F(x_1)+
  \leqslant-\frac{1}{2}\sum_{i=1}^k\alpha_i
  \mathbb{E}\big[\|\nabla F(x_k)\|_2^2\big]
  +\frac{LM_d}{2}\sum_{i=1}^k\alpha_i^2,
\end{equation*}
and we further obtain
\begin{equation*}
  \sum_{i=1}^k\alpha_i\mathbb{E}\big[\|\nabla F(x_k)\|_2^2\big]
  \leqslant2\big(F(x_1)-F_{\textrm{inf}}\big)+LM_d\sum_{i=1}^k\alpha_i^2,
\end{equation*}
thus, \cref{SGFD:eq:stepsize} implies that $\sum_{i=1}^\infty\alpha_i^2 <\infty$, further, the right-hand side of this inequality converges to a finite limit as $k$ tends to $\infty$; this implies $\liminf_{k\to\infty} \mathbb{E}[\|\nabla F(x_k)\|_2^2]=0$ since $\sum_{i=1}^\infty\alpha_i=\infty$, so the proof is complete.
\end{proof}

Now we could prove the following theorem which guarantees that the expected gradient norms converge to zero for the gradient-free method in nonconvex settings.
\begin{theorem}\label{SGFD:thm:nonconvex}
Suppose the conditions of \cref{SGFD:lem:nonconvex} hold. Then
\begin{equation*}
  \lim_{k\to\infty}\mathbb{E}\big[\|\nabla F(x_k)\|_2^2\big]=0.
\end{equation*}
\end{theorem}
\begin{proof}
From \cref{SGFD:ass:AG} and $\nabla\|\nabla F(x_k)\|_2^2=2\nabla^2 F(x_k)\nabla F(x_k)$, we have
\begin{align*}
  \|\nabla F(x_{k+1})\|_2^2-\|\nabla F(x_k)\|_2^2\leqslant&
  \big[\nabla\|\nabla F(x_k)\|_2^2\big]^\mathrm{T}(x_{k+1}-x_k)
  +\frac{L_G}{2}\|x_{k+1}-x_k\|_2^2 \\
  =&2\nabla^2 F(x_k)^\mathrm{T}\nabla F(x_k)^\mathrm{T}(x_{k+1}-x_k)
  +\frac{L_G}{2}\|x_{k+1}-x_k\|_2^2,
\end{align*}
noting that $x_{k+1}-x_k=s(x_k,\xi_k,\alpha_k,\zeta_k)$,
taking expectations with respect to the distribution of $\xi_k$ and $\zeta_k$, and using \cref{SGFD:ass:A1}, we get
\begin{align*}
  &\mathbb{E}_{\xi_k,\zeta_k}\big[\|\nabla F(x_{k+1})\|_2^2\big]
  -\|\nabla F(x_k)\|_2^2 \\
  \leqslant&2\nabla^2 F(x_k)^\mathrm{T}\nabla F(x_k)^\mathrm{T}
  \mathbb{E}_{\xi_k,\zeta_k}[s(x_k,\xi_k,\alpha_k,\zeta_k)]
  +\frac{L_G}{2}\mathbb{E}_{\xi_k,\zeta_k}
  \big[\|s(x_k,\xi_k,\alpha_k,\zeta_k)\|_2^2\big] \\
  \leqslant&2L\|\nabla F(x_k)\|_2
  \|\mathbb{E}_{\xi_k,\zeta_k}[s(x_k,\xi_k,\alpha_k,\zeta_k)]\|_2
  +\frac{L_G}{2}\mathbb{E}_{\xi_k,\zeta_k}
  \big[\|s(x_k,\xi_k,\alpha_k,\zeta_k)\|_2^2\big],
\end{align*}
together with \cref{SGFD:thm:step,SGFD:ass:A3}, we have
\begin{equation*}
  \mathbb{E}_{\xi_k,\zeta_k}\big[\|\nabla F(x_{k+1})\|_2^2\big]
  -\|\nabla F(x_k)\|_2^2 \leqslant2L\alpha_k\|\nabla F(x_k)\|_2^2
  +\frac{\alpha_k^2A}{4}\|\nabla F(x_k)\|_2^2+\frac{\alpha_k^2B}{8},
\end{equation*}
where $A=4L+3L_G+2L_GM_V$ and $B=4L^3d^3D_\zeta^2+3L_Gd^3D_\zeta^2+4L_GM$. And it follows from taking total expectations that
\begin{align*}
  \mathbb{E}[\|\nabla F(x_{k+1})\|_2^2]-\mathbb{E}[\|\nabla F(x_k)\|_2^2]
  \leqslant&2L\alpha_k\mathbb{E}[\|\nabla F(x_k)\|_2^2]+\frac{\alpha_k^2A}{4}
  \mathbb{E}[\|\nabla F(x_k)\|_2^2]+\frac{\alpha_k^2B}{8} \\
  \leqslant&\Big(2L+\frac{A}{4LM_G}\Big)
  \alpha_k\mathbb{E}[\|\nabla F(x_k)\|_2^2]+\frac{\alpha_k^2B}{8},
\end{align*}
by repeatedly applying the last inequality above through iteration from $1$ to $k$, and rearranging, we could get
\begin{equation*}
  \mathbb{E}[\|\nabla F(x_{k+1})\|_2^2]\leqslant
  \|\nabla F(x_1)\|_2^2+\Big(2L+\frac{A}{4LM_G}\Big)
  \sum_{i=1}^k\alpha_i\mathbb{E}[\|\nabla F(x_k)\|_2^2]
  +\frac{B}{8}\sum_{i=1}^k\alpha_k^2,
\end{equation*}
thus, it follows from \cref{SGFD:lem:nonconvex,SGFD:eq:stepsize} that $\mathbb{E}[\|\nabla F(x_k)\|_2^2]<\infty$ for every $k\in\mathbb{N}$. Further, if we let
\begin{align*}
  S_k^+=&\sum_{i=1}^{k-1}\max(0,\mathbb{E}[\|\nabla F(x_{k+1})\|_2^2]
  -\mathbb{E}[\|\nabla F(x_k)\|_2^2]) \\\textrm{and}~~
  S_k^-=&\sum_{i=1}^{k-1}\max(0,\mathbb{E}[\|\nabla F(x_k)\|_2^2]
  -\mathbb{E}[\|\nabla F(x_{k+1})\|_2^2]),
\end{align*}
then we have
\begin{equation*}
  \mathbb{E}[\|\nabla F(x_k)\|_2^2]=\|\nabla F(x_1)\|_2^2+S_k^+-S_k^-,
\end{equation*}
and it is clear that
\begin{equation*}
  S_k^+\leqslant\Big(2L+\frac{A}{4LM_G}\Big)
  \sum_{i=1}^k\alpha_i\mathbb{E}[\|\nabla F(x_k)\|_2^2]
  +\frac{B}{8}\sum_{i=1}^k\alpha_k^2<\infty,
\end{equation*}
that is, the nondecreasing sequence $\{S_k^+\}_k$ is upper bounded and therefore converges; further, note that $\mathbb{E}[\|\nabla F(x_k)\|_2^2]\geqslant0$ gives $S_k^-\leqslant\|\nabla F(x_1)\|_2^2+S_k^+$, we know that the nondecreasing sequence $\{S_k^-\}_k$ is upper bounded and therefore also converges. Hence, the sequence $\{\mathbb{E}[\|\nabla F(x_k)\|_2^2]\}_k$ converges, and further, $\liminf_{k\to\infty} \mathbb{E}[\|\nabla F(x_k)\|_2^2]=0$ (\cref{SGFD:lem:nonconvex}) means that
this limit must be zero, as we desired.
\end{proof}

\section{Accelerated method with momentum}
\label{SGFD:s4}

Stochastic gradient methods with momentum are very popular because of their practical performance in the community working on training deep neural
networks \cite{LeenT1993A_momentum,HintonG2013M_SG_DNN}. Now we will add a momentum term to our gradient-free methods. Especially, we provide a mean-variance analysis about the inclusion of momentum in stochastic settings. And it is shown that when employing a decaying stepsize $\alpha_k=\mathcal{O}(1/k)$, the momentum term we used adds a deviation of order $\mathcal{O}(1/k)$ but controls the variance at the order $\mathcal{O}(1/k)$ for the $k$th iteration. This make our accelerated method to achieve a convergence rate $\mathcal{O}(1/k^2)$.

\subsection{Methods}

Gradient-free methods with momentum are procedures in which each step is chosen as a weighted average of all historical stochastic directions. And specifically, with an initial point $x_1$, these methods are characterized by the iteration
\begin{equation}\label{SGFD:eq:m}
  x_{k+1}\leftarrow x_k+\alpha_k
  \frac{v_k}{\sum_{j=1}^k\prod_{l=j}^k\gamma(l)},
\end{equation}
where the direction vector $v_k$ could be recursively defined as
\begin{equation*}
  v_k=\gamma(k)v_{k-1}+\frac{1}{\alpha_k}s(x_k,\xi_k,\alpha_k,\zeta_k),
\end{equation*}
the changing decay factor $\gamma(k)\in(0,1)$ takes the form
\begin{equation}\label{SGFD:eq:cf}
  \gamma(k)=\Big(\frac{k}{k+1}\Big)^p,~~\textrm{where}~~p>0,
\end{equation}
and the stepsize sequence $\{\alpha_k\}$ takes the form 
\begin{equation}\label{SGFD:eq:stepsizem}
  \alpha_k=\frac{\beta}{k+\sigma}~~\textrm{for some}~~\beta>\frac{4}{l}~~
  \textrm{and}~~\sigma>0~~\textrm{such that}~~
  \alpha_1\leqslant\frac{1}{LM_{G,p}^{(1)}};
\end{equation}
here, the constant $M_{G,p}^{(k)}$ will be given in \cref{SGFD:lem:EmB}. 

Notice that $\prod_{l=j}^k\gamma(l)=\frac{j^p}{(k+1)^p}$, the direction vector could be rewritten as 
\begin{equation*}
  v_k=\frac{1}{(k+1)^p}\sum_{j=1}^k\frac{j^p}{\alpha_j}
  s(x_j,\xi_j,\alpha_j,\zeta_j).
\end{equation*}
Thus, these methods can be rewritten by the iteration
\begin{equation}\label{SGFD:eq:m2}
  x_{k+1}\leftarrow x_k+\alpha_km_k,
\end{equation}
where the weighted average direction
\begin{equation}\label{SGFD:eq:stepm}
  m_k=\frac{v_k}{\sum_{i=1}^k\prod_{l=i}^k\gamma(l)}
  =\frac{1}{\sum_{i=1}^ki^p}\sum_{j=1}^k
  \frac{j^p}{\alpha_j}s(x_k+\delta_{j,k},\xi_j,\alpha_j,\zeta_j),
\end{equation}
where $\delta_{j,k}=x_j-x_k$ for all $1\leqslant j\leqslant k$. And notice that the accelerated method with momentum and the SGFD method are exactly the same in the first iteration.

\subsection{Expectation analysis}

It is very reasonable to expect that the accelerated method with momentum could also reach the sublinear convergence rate $\mathcal{O}(1/k)$ under a strong convexity condition, like the SGFD method. 

Actually, since the accelerated method is the same as the SGFD method in the first iteration, we assume, without loss of generality, that the first $k$ iterations $\{x_j\}_{j=1}^k$ generated by \cref{SGFD:eq:m2} has the sublinear convergence rate under \cref{SGFD:ass:A1,SGFD:ass:ASC,SGFD:ass:A2,SGFD:ass:A3}, that is, for every $1\leqslant j\leqslant k$, we have 
\begin{equation}\label{SGFD:eq:CRf}
  \mathbb{E}[F(x_j)]-F_*=\mathcal{O}(1/k);
\end{equation}
or equivalently, by noting that \cref{SGFD:ass:ASC}, we have 
\begin{equation}\label{SGFD:eq:CRx}
  \mathbb{E}[\|x_j-x_*\|_2^2]\leqslant2l^{-1}
  (\mathbb{E}[F(x_j)]-F_*)=\mathcal{O}\big(j^{-1}\big);
\end{equation}
and further, for every $1\leqslant j\leqslant k$, it follows from \cref{SGFD:eq:CRx} that
\begin{equation}\label{SGFD:eq:CEd}
  \mathbb{E}[\|\delta_{j,k}\|_2^2]=\mathbb{E}[\|x_j-x_k\|_2^2]
  \leqslant\mathbb{E}[\|x_j-x_*\|_2+\|x_k-x_*\|_2]^2
  =\mathcal{O}\big(j^{-1}\big).
\end{equation}
And then, $\mathbb{V}[\delta_{j,k}]>0$ from the use of stachastic directions, so we have
\begin{equation}\label{SGFD:eq:delta}
  \|\mathbb{E}[\delta_{j,k}]\|_2^2<
  \|\mathbb{E}[\delta_{j,k}]\|_2^2+\mathbb{V}[\delta_{j,k}]=
  \mathbb{E}[\|\delta_{j,k}\|_2^2]=\mathcal{O}\big(j^{-1}\big).
\end{equation}
that is, $\|\mathbb{E}[\delta_{j,k}]\|_2$ tends to zero strictly faster than $\mathcal{O}(j^{-\frac{1}{2}})$. Hence, we introduce the following assumption.
\begin{assumption}\label{SGFD:ass:Order}
The sequence of iterates $\{x_j\}_{j=1}^k$ satisfy for every $k\in\mathbb{N}$ and $1\leqslant j\leqslant k$, there is a fixed $\kappa>0$ such that
\begin{equation}\label{SGFD:eq:Order}
  \|\mathbb{E}[\delta_{j,k}]\|_2^2=\mathcal{O}
  \Big(j^{-\kappa}\mathbb{E}[\|\delta_{j,k}\|_2^2]\Big).
\end{equation}
\end{assumption}

According to \cref{SGFD:eq:CEd,SGFD:eq:Order}, we obtain 
\begin{equation}\label{SGFD:eq:REC1}
  \|\mathbb{E}[\delta_{j,k}]\|_2=\mathcal{O}\Big(j^{-\frac{1+\kappa}{2}}\Big).
\end{equation}
We will finally show that \cref{SGFD:ass:Order} implies actually $\|\mathbb{E}[\delta_{j,k}]\|_2=\mathcal{O}(j^{-1})$ in \cref{SGFD:s4:5}. 

In the rest of \cref{SGFD:s4}, we shall prove by induction on $k$ that the acceleration method maintain the sublinear convergence rate $\mathcal{O}(1/k)$ under a strong convexity condition; and furthermore, we shall also prove that it can achieve a convergence rate $\mathcal{O}(1/k^2)$ under \cref{SGFD:ass:Order}. 

Now we prove two lemmas which are necessary for the mean-variance analysis.
\begin{lemma}\label{SGFD:lem:mdecay1}
Under the conditions of \cref{SGFD:eq:CRf}, suppose that the sequence of iterates $\{x_j\}_{j=1}^k$ is generated by \cref{SGFD:eq:m2} with a stepsize sequence $\{\alpha_k\}$ taking the form \cref{SGFD:eq:stepsizem} and a changing decay factor $\gamma(k)$ taking the form \cref{SGFD:eq:cf}. Then, there is $\tau_p<\infty$ such that 
\begin{equation}\label{SGFD:eq:mdecayA}
  \frac{1}{\alpha_k\sum_{i=1}^k\prod_{l=i}^k\gamma(l)}
  \sum_{j=1}^k\alpha_j\prod_{l=j}^k\gamma(l)\leqslant\tau_p,
\end{equation}
and there is $D'_p<\infty$ such that for any given diagonal matrix $\Lambda= \mathrm{diag}(\lambda_1,\cdots,\lambda_d)$ with $\lambda_i\in [-L,L]$, the inequality 
\begin{equation}\label{SGFD:eq:mdecayB}
  \frac{1}{\sqrt{\alpha_k}\sum_{i=1}^k\prod_{l=i}^k\gamma(l)}
  \bigg\|\sum_{j=1}^k\Lambda\delta_{j,k}\prod_{l=j}^k\gamma(l)\bigg\|_2
  \leqslant\frac{LD'_p}{2}
\end{equation}
holds in probability.
\end{lemma}
\begin{proof}
According to \cref{SGFD:eq:cf}, we have $\prod_{l=j}^k\gamma(l)=\frac{j^p}{(k+1)^p}$; by further noting that $\alpha_j=\mathcal{O}(1/j)$, it follows that 
\begin{equation*}
  \sum_{j=1}^k\alpha_j\prod_{l=j}^k\gamma(l)
  =\frac{\sum_{j=1}^k\alpha_jj^p}{(k+1)^p}
  =\mathcal{O}\bigg(\frac{\sum_{j=1}^kj^{p-1}}{(k+1)^p}\bigg)
  =\mathcal{O}\bigg(\frac{k^p}{(k+1)^p}\bigg)=\mathcal{O}(1),
\end{equation*}
and 
\begin{equation}\label{SGFD:eq:decaylocal1}
  \alpha_k\sum_{j=1}^k\prod_{l=j}^k\gamma(l)
  =\frac{\sum_{j=1}^kj^p}{\alpha_k(k+1)^p}
  =\mathcal{O}\bigg(\frac{\sum_{j=1}^kj^p}{(k+1)^{p+1}}\bigg)
  =\mathcal{O}\bigg(\frac{k^{p+1}}{(k+1)^{p+1}}\bigg)=\mathcal{O}(1),
\end{equation}
these two results above yield the inequality \cref{SGFD:eq:mdecayA}. 

Similar to \cref{SGFD:eq:decaylocal1}, we also have 
$\sqrt{\alpha_k}\sum_{j=1}^k\prod_{l=j}^k\gamma(l)=\mathcal{O} \big(k^{\frac{1}{2}}\big)$; thus, to prove \cref{SGFD:eq:mdecayB}, we only need to show that
\begin{equation*}
  \bigg\|\sum_{j=1}^k\Lambda\delta_{j,k}\prod_{l=j}^k\gamma(l)\bigg\|_2
  =\frac{\|\sum_{j=1}^k\Lambda\delta_{j,k}j^p\|_2}{(k+1)^p}
  =\mathcal{O}\big(k^{\frac{1}{2}}\big)~~~\textrm{or}~~~\bigg\|\sum_{j=1}^k
  \Lambda\delta_{j,k}j^p\bigg\|_2=\mathcal{O}\big(k^{p+\frac{1}{2}}\big)
\end{equation*}
holds in probability. We use the mean-variance analysis: it follows from \cref{SGFD:eq:delta} that
\begin{equation*}
  \mathbb{E}\bigg[\sum_{j=1}^k\Lambda\delta_{j,k}j^p\bigg]
  =\Lambda\sum_{j=1}^kj^{p}~\mathbb{E}[\delta_{j,k}]
  =\mathcal{O}\bigg(\Lambda\sum_{j=1}^kj^{p-\frac{1}{2}}\bigg)
  =\mathcal{O}\big(k^{p+\frac{1}{2}}\big);
\end{equation*}
meanwhile, according to \cref{SGFD:eq:delta}, we also have 
\begin{equation*}
  \mathbb{V}\bigg[\sum_{j=1}^k\Lambda\delta_{j,k}j^p\bigg]
  =\Lambda^2\sum_{j=1}^kj^{2p}~\mathbb{V}[\delta_{j,k}]
  \leqslant\Lambda^2\sum_{j=1}^kj^{2p-1}=\mathcal{O}(k^{2p}).
\end{equation*}
Using Chebyshev's inequality, there is a $C>0$ such that for $\epsilon>0$,
\begin{equation}\label{SGFD:eq:localV}
  \mathbb{P}\bigg(\bigg\|\sum_{j=1}^k\Lambda\delta_{j,k}j^p
  -Ck^{p+\frac{1}{2}}\Lambda\bigg\|_2\geqslant\epsilon
  k^p\Lambda\bigg)\leqslant\frac{1}{\epsilon^2},
\end{equation}
which gives the inequality \cref{SGFD:eq:mdecayB} in probability. 
\end{proof}

Under \cref{SGFD:ass:Order}, it is clear that \cref{SGFD:eq:mdecayB} could be further strengthened.
\begin{lemma}\label{SGFD:lem:mdecay2}
Suppose the conditions of \cref{SGFD:lem:mdecay1,SGFD:ass:Order} hold. Then, there is $D_p<\infty$ such that for any given diagonal matrix $\Lambda= \mathrm{diag}(\lambda_1,\cdots,\lambda_d)$ with $\lambda_i\in [-L,L]$, the inequality
\begin{equation}\label{SGFD:eq:mdecayC}
  \frac{1}{\alpha_k^s\sum_{i=1}^k
  \prod_{l=i}^k\gamma(l)}\bigg\|\sum_{j=1}^k\Lambda\delta_{j,k}
  \prod_{l=j}^k\gamma(l)\bigg\|_2\leqslant\frac{LD_p}{2}
\end{equation}
holds in probability, where $s=\min\big(1,\frac{1+\kappa}{2}\big)$. 
\end{lemma}
\begin{proof}
Similar to \cref{SGFD:eq:decaylocal1}, we have
$\alpha_k^{\frac{1+\kappa}{2}}\sum_{j=1}^k\prod_{l=j}^k\gamma(l)=\mathcal{O} \big(k^{\frac{1-\kappa}{2}}\big)$; thus, to prove \cref{SGFD:eq:mdecayC}, we only need to show that
\begin{equation*}
  \bigg\|\sum_{j=1}^k\Lambda\delta_{j,k}\prod_{l=j}^k\gamma(l)\bigg\|_2
  \!\!=\!\frac{\|\sum_{j=1}^k\Lambda\delta_{j,k}j^p\|_2}{(k+1)^p}
  =\mathcal{O}\big(k^{\frac{1-\kappa}{2}}\big)~~~\textrm{or}~~~
  \bigg\|\sum_{j=1}^k\Lambda\delta_{j,k}j^p\bigg\|_2
  \!\!=\!\mathcal{O}\big(k^{p+\frac{1-\kappa}{2}}\big)
\end{equation*}
holds in probability. First, 
\begin{equation*}
  \mathbb{E}\bigg[\sum_{j=1}^k\Lambda\delta_{j,k}j^p\bigg]
  =\Lambda\sum_{j=1}^kj^{p}~\mathbb{E}[\delta_{j,k}]
  =\mathcal{O}\bigg(\Lambda\sum_{j=1}^kj^{p-\frac{1+\kappa}{2}}\bigg)
  =\big(k^{p+\frac{1-\kappa}{2}}\big);
\end{equation*}
together with \cref{SGFD:eq:localV} and using Chebyshev's inequality, there is a $C>0$ such that
\begin{equation*}
  \mathbb{P}\bigg(\bigg\|\sum_{j=1}^k\delta_{j,k}j^p
  -Ck^{p+\frac{1-\kappa}{2}}\Lambda\bigg\|_2
  \geqslant\epsilon Vk^p\Lambda\bigg)\leqslant\frac{1}{\epsilon^2}.
\end{equation*}
It is worth noting that when $\kappa\geqslant1$, the variance will become the principal part; so the proof is complete.
\end{proof}

From \cref{SGFD:lem:mdecay1,SGFD:lem:mdecay2} we can get different deviations of the expectation.
\begin{theorem}\label{SGFD:thm:momentum}
Suppose the conditions of \cref{SGFD:lem:mdecay1} hold. Then the weighted average directions \cref{SGFD:eq:stepm} satisfy the asymptotic unbiasedness
\begin{equation}\label{SGFD:eq:mlimit}
  \lim_{\alpha_k\to0}\mathbb{E}[m_k]=-\nabla F(x_k),
\end{equation}
and for every $k\in\mathbb{N}$, it follows that
\begin{align}
  \|\mathbb{E}[m_k]\|_2\leqslant&\|\nabla F(x_k)\|_2+
  \frac{\sqrt{\alpha_k}LD'_p+
  \alpha_kLd^{\frac{3}{2}}D_\zeta\tau_p}{2}
  ~~\textrm{and}~~\label{SGFD:eq:mstep1} \\
  \nabla F(x_k)^\mathrm{T}\mathbb{E}[m_k]
  \leqslant&-\|\nabla F(x_k)\|_2^2+\frac{\sqrt{\alpha_k}LD'_p
  +\alpha_kLd^{\frac{3}{2}}D_\zeta\tau_p}{2}\|\nabla F(x_k)\|_2 
  \label{SGFD:eq:mstep2}
\end{align}
in probability, where the constant $D_\zeta$ comes from \cref{SGFD:lem:E13}.
\end{theorem}
\begin{proof}
Let $\Delta_{j,k}$ denote the difference of $\nabla F(x_j)$ and $\nabla F(x_k)$, i.e., 
\begin{equation*}
  \Delta_{j,k}=\nabla F(x_j)-\nabla F(x_k)=
  \nabla F(x_k+\delta_{j,k})-\nabla F(x_k),
\end{equation*}
then according to \cref{SGFD:ass:A1}, 
\begin{equation}\label{SGFD:eq:localD}
  \|\Delta_{j,k}\|_\infty\leqslant\|\Delta_{j,k}\|_2\leqslant 
  L\|\delta_{j,k}\|_2;
\end{equation}
therefore, there is a diagonal matrix $\Lambda= \mathrm{diag}(\lambda_1,\cdots,\lambda_d)$ with $\lambda_i\in [-L,L]$ such that $\Delta_{j,k}=\Lambda\delta_{j,k}$. Then, along with \cref{SGFD:lem:Estep}, there are $\{C_j\}_{j=1}^k$ with $C_j\leqslant L$ such that
\begin{align*}
  \mathbb{E}[m_k]=&\frac{1}{\sum_{i=1}^k\prod_{l=i}^k\gamma(l)}
  \sum_{j=1}^k\frac{\prod_{l=j}^k\gamma(l)}{\alpha_j}
  \mathbb{E}_{\xi_j,\zeta_j}[s(x_j,\xi_j,\alpha_j,\zeta_j)] \\
  =&-\frac{1}{\sum_{i=1}^k\prod_{l=i}^k\gamma(l)}
  \sum_{j=1}^k\bigg[\prod_{l=j}^k\gamma(l)\bigg(\nabla F(x_j)+\frac{\alpha_j}{2}
  \mathbb{E}_{\zeta_j}\!\!\left[C_j
  \Big(\zeta_j^\mathrm{T}\zeta_j\Big)\zeta_j\right]\bigg)\bigg] \\
  =&-\frac{1}{\sum_{i=1}^k\prod_{l=i}^k\gamma(l)}
  \sum_{j=1}^k\bigg[\prod_{l=j}^k\gamma(l)\bigg(\nabla F(x_k)
  +\Lambda\delta_{j,k}+\frac{\alpha_j}{2}\mathbb{E}_{\zeta_j}\!\!
  \left[C_j\Big(\zeta_j^\mathrm{T}\zeta_j\Big)\zeta_j\right]\bigg)\bigg] \\
  =&-\nabla F(x_k)-\frac{1}{\sum_{i=1}^k\prod_{l=i}^k\gamma(l)}
  \sum_{j=1}^k\bigg[\prod_{l=j}^k\gamma(l)\bigg(
  \Lambda\delta_{j,k}+\frac{\alpha_j}{2}\mathbb{E}_{\zeta_j}\!\!
  \left[C_j\Big(\zeta_j^\mathrm{T}\zeta_j\Big)\zeta_j\right]\bigg)\bigg],
\end{align*}
and we rewrite it further as
\begin{equation}\label{SGFD:eq:localmkD}
  \mathbb{E}[m_k]=-\nabla F(x_k)-R_\delta-R,
\end{equation}
where two vectors $R_\delta,R\in\mathbb{R}^d$ are given as
\begin{equation*}
  R_\delta:=\frac{1}{\sum_{i=1}^k\prod_{l=i}^k\gamma(l)}
  \sum_{j=1}^k\bigg(\prod_{l=j}^k\gamma(l)\Lambda\delta_{j,k}\bigg)
\end{equation*}
and
\begin{equation*}
  R:=\frac{1}{\sum_{i=1}^k\prod_{l=i}^k\gamma(l)}
  \sum_{j=1}^k\bigg(\frac{\alpha_j\prod_{l=j}^k\gamma(l)}{2}
  \mathbb{E}_{\zeta_j}\!\!\left[C_j
  \Big(\zeta_j^\mathrm{T}\zeta_j\Big)\zeta_j\right]\!\bigg).
\end{equation*}

First, it follows from \cref{SGFD:lem:E13,SGFD:lem:mdecay1} that
\begin{equation*}
  \|R\|_2\leqslant\frac{\alpha_kLd^{\frac{3}{2}}D_\zeta}{2}
  \frac{1}{\alpha_k\sum_{i=1}^k\prod_{l=i}^k\gamma(l)}
  \sum_{j=1}^k\alpha_j\prod_{l=j}^k\gamma(l)\leqslant
  \frac{\alpha_kLd^{\frac{3}{2}}D_\zeta\tau_\gamma}{2}.
\end{equation*}
and second, according to \cref{SGFD:lem:mdecay1}, we have 
\begin{equation*}
  \|R_\delta\|_2\leqslant\frac{\sqrt{\alpha_k}}
  {\sqrt{\alpha_k}\sum_{i=1}^k\prod_{l=i}^k\gamma(l)}
  \bigg\|\sum_{j=1}^k\Lambda\delta_{j,k}\prod_{l=j}^k\gamma(l)\bigg\|_2
  \leqslant\frac{\sqrt{\alpha_k}LD'_p}{2};
\end{equation*}
Thus, one obtains \cref{SGFD:eq:mlimit,SGFD:eq:mstep1} by noting that
\begin{align*}
  \|\mathbb{E}[m_k]+\nabla F(x_k)\|_\infty\leqslant&
  \|R_\delta\|_\infty+\|R\|_\infty\leqslant\|R_\delta\|_2+\|R\|_2
  ~~\textrm{and}~~ \\
  \|\mathbb{E}[m_k]\|_2\leqslant&\|\nabla F(x_k)\|_2+\|R_\delta\|_2+\|R\|_2;
\end{align*}
and further obtains \cref{SGFD:eq:mstep2} by noting that
\begin{align*}
  \nabla F(x_k)^\mathrm{T}\mathbb{E}[m_k]=&
  -\|\nabla F(x_k)\|_2^2-\nabla F(x_k)^\mathrm{T}(R_\delta+R) \\
  \leqslant&-\|\nabla F(x_k)\|_2^2
  +\|\nabla F(x_k)\|_2(\|R_\delta\|_2+\|R\|_2),
\end{align*}
and the proof is complete.
\end{proof}

Under \cref{SGFD:ass:Order}, \cref{SGFD:eq:mstep1,SGFD:eq:mstep2} can be further improved.
\begin{theorem}\label{SGFD:thm:momentum2}
Suppose the conditions of \cref{SGFD:lem:mdecay2} hold. Then for every $k\in\mathbb{N}$, the following conditions
\begin{align}
  \|\mathbb{E}[m_k]\|_2\leqslant&\|\nabla F(x_k)\|_2+
  \frac{\alpha_k^sLD_p+\alpha_kLd^{\frac{3}{2}}D_\zeta\tau_p}{2}
  ~~\textrm{and}~~\label{SGFD:eq:mstep3} \\
  \nabla F(x_k)^\mathrm{T}\mathbb{E}[m_k]
  \leqslant&-\|\nabla F(x_k)\|_2^2+
  \frac{\alpha_k^sLD_p+\alpha_kLd^{\frac{3}{2}}D_\zeta\tau_p}{2}
  \|\nabla F(x_k)\|_2 \label{SGFD:eq:mstep4}
\end{align}
hold in probability, where $s=\min\big(1,\frac{1+\kappa}{2}\big)$ and the constant $D_\zeta$ comes from \cref{SGFD:lem:E13}.
\end{theorem}
\begin{proof}
It follows from \cref{SGFD:eq:localD,SGFD:lem:mdecay2} that there exists a diagonal matrix $\Lambda= \mathrm{diag}(\lambda_1,\cdots,\lambda_d)$ with $\lambda_i\in [-L,L]$
\begin{equation*}
  \|R_\delta\|_2\leqslant\frac{\alpha_k^s}
  {\alpha_k^s\sum_{i=1}^k\prod_{l=i}^k\gamma(l)}
  \bigg\|\sum_{j=1}^k\Lambda\delta_{j,k}\prod_{l=j}^k\gamma(l)\bigg\|_2
  \leqslant\frac{\alpha_k^sLD_p}{2}.
\end{equation*}
Recalling \cref{SGFD:eq:localmkD} and $\|R\|_2\leqslant \frac{\alpha_kLd^{\frac{3}{2}}D_\zeta\tau_\gamma}{2}$, then yields \cref{SGFD:eq:mstep3,SGFD:eq:mstep4}. 
\end{proof}

\subsection{Variance analysis}

As an important result, we will show that the changing decay factor \cref{SGFD:eq:cf} could reduce the variance of $m_k$ to zero with a rate $\mathcal{O}(k^{-1})$.
\begin{lemma}\label{SGFD:lem:momentumV}
Under the conditions of \cref{SGFD:ass:A3}, suppose that (i) suppose that the sequence of iterates $\{x_j\}_{j=1}^k$ is generated by \cref{SGFD:eq:m2} with a stepsize sequence $\{\alpha_k\}$ taking the form \cref{SGFD:eq:stepsizem} and a changing decay factor $\gamma(k)$ taking the form \cref{SGFD:eq:cf}, and (ii) the sequence $\{x_k\}$ satisfies $\|x_i-x_j\|_2\leqslant D$ for any $i,j\in\mathbb{N}$, then
\begin{equation}\label{SGFD:eq:momentumVC}
  \mathbb{V}[m_k]\leqslant C_p\alpha_k
  \Big(M+2M_V\|\nabla F(x_k)\|_2^2+2M_VL^2D^2\Big)
\end{equation}
where $C_p$ is a positive real constant.
\end{lemma}
\begin{remark}
If we consider a fixed decay factor $\gamma\in(0,1)$, then 
\begin{align*}
  \frac{\sum_{i=1}^k\big(\prod_{l=i}^k\gamma(l)\big)^2}
  {\big(\sum_{i=1}^k\prod_{l=i}^k\gamma(l)\big)^2}=
  \frac{\sum_{j=1}^k\gamma^{2(k-j)}}{\big(\sum_{i=1}^k\gamma^{k-i}\big)^2} 
  =\frac{1-\gamma^{2k}}{1-\gamma^2}\frac{(1-\gamma)^2}{(1-\gamma^k)^2}
  =\frac{1-\gamma}{1+\gamma}\frac{1+\gamma^k}{1-\gamma^k}.
\end{align*}
Further note that $\frac{1-\gamma}{1+\gamma}\frac{1+\gamma^k}{1-\gamma^k}$ decays to $\frac{1-\gamma}{1+\gamma}$ as $k$ increases for $0<\gamma<1$, so in this case, the variance of $m_k$ could be finally reduced to
\begin{equation*}
  \frac{1-\gamma}{1+\gamma}\Big(M+2M_V\|\nabla F(x_k)\|_2^2+2M_VL^2D^2\Big).
\end{equation*}
\end{remark}
\begin{proof}
According to \cref{SGFD:ass:A1,SGFD:eq:localD}, there is a diagonal matrix $\Lambda= \mathrm{diag}(\lambda_1,\cdots,\lambda_d)$ with $\lambda_i\in [-L,L]$ such that 
\begin{align*}
  \nabla F(x_j)=\nabla F(x_k+\delta_{j,k})=\nabla F(x_k)+\Lambda\delta_{j,k},
\end{align*}
where $x_j=x_k+\delta_{j,k}$; and by further noting that $\|\delta_{j,k}\|_2=\|x_j-x_k\|_2\leqslant D$, we have
\begin{align*}
  \|\nabla F(x_j)\|_2^2\leqslant&\Big(\|\nabla F(x_k)\|_2
  +L\|\delta_{j,k}\|_2\Big)^2 \\
  \leqslant&\|\nabla F(x_k)\|_2^2+L^2\|\delta_{j,k}\|_2^2
  +2L\|\delta_{j,k}\|_2\|\nabla F(x_k)\|_2 \\
  \leqslant&2\|\nabla F(x_k)\|_2^2+2L^2\|\delta_{j,k}\|_2^2
  \leqslant2\|\nabla F(x_k)\|_2^2+2L^2D^2.
\end{align*}
Hence, along with \cref{SGFD:ass:A3}, we obtain
\begin{align*}
  \mathbb{V}[m_k]=&\frac{1}{\big(\sum_{i=1}^k\prod_{l=i}^k\gamma(l)\big)^2}
  \sum_{j=1}^k\frac{\big(\prod_{l=i}^k\gamma(l)\big)^2}{\alpha_j^2}
  \mathbb{V}_{\xi_j,\zeta_j}[s(x_j,\xi_j,\alpha_j,\zeta_j)] \\
  \leqslant&\frac{1}{\big(\sum_{i=1}^k\prod_{l=i}^k\gamma(l)\big)^2}
  \sum_{j=1}^k\bigg(\prod_{l=i}^k\gamma(l)\bigg)^2
  \Big(M+M_V\|\nabla F(x_j)\|_2^2\Big) \\
  \leqslant&\frac{\sum_{i=1}^k\big(\prod_{l=i}^k\gamma(l)\big)^2}
  {\big(\sum_{i=1}^k\prod_{l=i}^k\gamma(l)\big)^2}
  \Big(M+2M_V\|\nabla F(x_k)\|_2^2+2M_VL^2D^2\Big).
\end{align*}
According to \cref{SGFD:eq:cf}, $\prod_{l=i}^k\gamma(l)=\frac{i^p}{(k+1)^p}$. Since $0<\frac{i^p}{(k+1)^p}<1$ for any $p>0$ and every $i=1,\cdots,k$, we have $\frac{i^{2p}}{(k+1)^{2p}}<\frac{i^p}{(k+1)^p}$, then it follows that
\begin{equation*}
  \frac{\sum_{i=1}^k\frac{i^{2p}}{(k+1)^{2p}}}
  {\Big(\sum_{i=1}^k\frac{i^p}{(k+1)^p}\Big)^2}
  <\frac{\sum_{i=1}^k\frac{i^p}{(k+1)^p}}
  {\Big(\sum_{i=1}^k\frac{i^p}{(k+1)^p}\Big)^2}
  =\frac{1}{\sum_{i=1}^k\frac{i^p}{(k+1)^p}},
\end{equation*}
by further noting that
\begin{equation*}
  \sum_{i=1}^k\frac{i^p}{(k+1)^p}
  =\frac{1}{(k+1)^p}\sum_{i=1}^ki^p=\mathcal{O}(k)~~\textrm{for any}~~p>0.
\end{equation*}
this gives $\mathbb{V}[m_k]=\mathcal{O}(k^{-1})=\mathcal{O}(\alpha_k)$, and the proof is complete.
\end{proof}

Combining \cref{SGFD:thm:momentum,SGFD:thm:momentum2} and \cref{SGFD:lem:momentumV}, we can obtain two different bounds for each iteration of the accelerated method. 
\begin{lemma}\label{SGFD:lem:EmB}
Under the conditions of \cref{SGFD:thm:momentum,SGFD:lem:momentumV}, suppose that the stepsize sequence $\{\alpha_k\}$ satisfies $\alpha_k\leqslant\frac{1}{L}$. Then, the inequality
\begin{equation*}
  \mathbb{E}_{\xi_k,\zeta_k}[F(x_{k+1})]-F(x_k)\leqslant
  -\frac{3\alpha_k}{4}\|\nabla F(x_k)\|_2^2
  +\frac{\alpha_k^2LM_{G,p}^{(k)}}{2}\|\nabla F(x_k)\|_2^2
  +\frac{\alpha_k^2LM_{d,p,1}^{(k)}}{2}
\end{equation*}
holds in probability, where $M_{G,p}^{(k)}=\frac{3}{2}+2\alpha_kC_pM_V$ and $M_{d,p,1}^{(k)}=\frac{5L{D''_p}^2}{4}+\alpha_kC_p(M+2M_VL^2D^2)$; further,
\begin{equation*}
  \lim_{k\to\infty}M_{G,p}^{(k)}=\frac{3}{2}
  ~~\textrm{and}~~\lim_{k\to\infty}M_{d,p,1}^{(k)}=\frac{5L{D''_p}^2}{4}.
\end{equation*}
\end{lemma}
\begin{proof}
According to \cref{SGFD:eq:mstep1}, there is a $D''_p>D'_p$ such that 
\begin{equation*}
  \mathbb{E}_{\xi_k,\zeta_k}[m_k]\leqslant\|\nabla F(x_k)\|_2+
  \frac{\sqrt{\alpha_k}LD''_p}{2},
\end{equation*}
together with the Arithmetic Mean Geometric Mean inequality, one obtains
\begin{align*}
  \left\|\mathbb{E}_{\xi_k,\zeta_k}[m_k]\right\|_2^2
  \leqslant&\bigg(\|\nabla F(x_k)\|_2+
  \frac{\sqrt{\alpha_k}LD''_p}{2}\bigg)^2 \\
  \leqslant&\|\nabla F(x_k)\|_2^2
  +\sqrt{\alpha_k}LD''_p\|\nabla F(x_k)\|_2
  +\frac{\alpha_kL^2{D''_p}^2}{4} \\
  \leqslant&\frac{3}{2}\|\nabla F(x_k)\|_2^2
  +\frac{3\alpha_kL^2{D''_p}^2}{4}.
\end{align*}
Similarly, by \cref{SGFD:eq:mstep2} and the Arithmetic Mean Geometric Mean inequality, it holds that
\begin{align*}
  \nabla F(x_k)^\mathrm{T}\mathbb{E}_{\xi_k,\zeta_k}[m_k]
  \leqslant&-\|\nabla F(x_k)\|_2^2+
  \frac{\sqrt{\alpha_k}LD''_p}{2}\|\nabla F(x_k)\|_2 \\
  \leqslant&-\frac{3}{4}\|\nabla F(x_k)\|_2^2
  +\frac{\alpha_kL^2{D''_p}^2}{4}.
\end{align*}
Finally, by \cref{SGFD:lem:AL1} and \cref{SGFD:ass:A3}, and $\alpha_k\leqslant\frac{1}{L}$, the iterates satisfy
\begin{align*}
  \mathbb{E}_{\xi_k,\zeta_k}\![F(x_{k+1})]\!-\!F(x_k)\leqslant&
  \nabla F(x_k)^\mathrm{T}\mathbb{E}_{\xi_k,\zeta_k}\![\alpha_km_k]
  +\!\frac{L}{2}\|\mathbb{E}_{\xi_k,\zeta_k}\![\alpha_km_k]\|_2^2\!
  +\!\frac{L}{2}\mathbb{V}_{\xi_k,\zeta_k}\![\alpha_km_k] \\
  =&\alpha_k\nabla F(x_k)^\mathrm{T}\mathbb{E}_{\xi_k,\zeta_k}\![m_k]
  +\!\frac{\alpha_k^2L}{2}\|\mathbb{E}_{\xi_k,\zeta_k}\![m_k]\|_2^2\!
  +\!\frac{\alpha_k^2L}{2}\mathbb{V}_{\xi_k,\zeta_k}\![m_k] \\
  \leqslant&-\frac{3\alpha_k}{4}\|\nabla F(x_k)\|_2^2
  +\alpha_k^2L\left(\frac{3}{4}+\alpha_kC_pM_V\right)\|\nabla F(x_k)\|_2^2 \\
  &+\frac{\alpha_k^2L}{2}\bigg(\frac{5L{D''_p}^2}{4}
  +\alpha_kC_p(M+2M_VL^2D^2)\bigg) \\
  =&-\alpha_k\|\nabla F(x_k)\|_2^2
  +\frac{\alpha_k^2LM_{G,p}^{(k)}}{2}\|\nabla F(x_k)\|_2^2
  +\frac{\alpha_k^2LM_{d,p,1}^{(k)}}{2},
\end{align*}
and the proof is complete.
\end{proof}

From \cref{SGFD:thm:momentum2}, the following lemma could be proved in the same way as the proof of \cref{SGFD:lem:EmB}.
\begin{lemma}\label{SGFD:lem:EmB2}
Under the conditions of \cref{SGFD:thm:momentum2,SGFD:lem:momentumV}, suppose that the stepsize sequence $\{\alpha_k\}$ satisfies $\alpha_k\leqslant\frac{1}{L}$. Then, the inequality 
\begin{equation*}
  \mathbb{E}_{\xi_k,\zeta_k}[F(x_{k+1})]-F(x_k)\leqslant
  -\frac{3\alpha_k}{4}\|\nabla F(x_k)\|_2^2
  +\frac{\alpha_k^2LM_{G,p}^{(k)}}{2}\|\nabla F(x_k)\|_2^2
  +\frac{\alpha_k^{2+\kappa'}LM_{d,p,2}^{(k)}}{2}
\end{equation*}
holds in probability, where $\kappa'=\min(1,\kappa)$, the constant $M_{G,p}^{(k)}$ comes from \cref{SGFD:lem:EmB},
\begin{equation*}
  M_{d,p,2}^{(k)}=\frac{5}{4}L{D'''_p}^2(D_p+d^{\frac{3}{2}}D_\zeta\tau_p)^2
  +C_p\alpha_k^{\frac{1-\kappa'}{2}}(M+2M_VL^2D^2),
\end{equation*}
and $D'''_p>D_p$ is a constant such that $\mathbb{E}_{\xi_k,\zeta_k}[m_k] \leqslant\|\nabla F(x_k)\|_2+\alpha_k^\frac{1+\kappa'}{2}LD'''_p/2$. 
\end{lemma}

\subsection{Average behavior of iterations}

According to \cref{SGFD:lem:EmB,SGFD:lem:EmB2}, it is also easy to analyze the average behavior of iterations of the accelerated method for strong convex functions. And the following theorems could be proved in the same way as the proof of \cref{SGFD:thm:AB}.
\begin{theorem}\label{SGFD:thm:ABm}
Under the conditions of \cref{SGFD:lem:EmB} and \cref{SGFD:ass:ASC}, suppose that the stepsize sequence $\{\alpha_k\}$ satisfies $\alpha_k\leqslant\frac{1}{LM_{G,p}^{(1)}}$. Then, the inequality
\begin{equation*}
  \mathbb{E}[F(x_{k+1})-F_*]\leqslant
  \left[\prod_{i=1}^k\left(1-\frac{\alpha_il}{2}\right)\right]
  [F(x_1)-F_*]+\frac{LM_{d,p,1}}{2}\sum_{i=1}^k
  \alpha_i^2\prod_{j=i+1}^k\left(1-\frac{\alpha_il}{2}\right)
\end{equation*}
holds in probability, where the constants $M_{G,p}^{(k)}$ and $M_{d,p,1}$ come from \cref{SGFD:lem:EmB}.
\end{theorem}

\begin{theorem}\label{SGFD:thm:ABm2}
Under the conditions of \cref{SGFD:lem:EmB2} and \cref{SGFD:ass:ASC}, suppose that the stepsize sequence $\{\alpha_k\}$ satisfies $\alpha_k\leqslant\frac{1}{LM_{G,p}^{(1)}}$. Then, the inequality
\begin{equation*}
  \mathbb{E}[F(x_{k+1})-F_*]\leqslant
  \left[\prod_{i=1}^k\left(1-\frac{\alpha_il}{2}\right)\right]
  [F(x_1)-F_*]+\frac{LM_{d,p,2}}{2}\sum_{i=1}^k
  \alpha_i^{2+\kappa'}\prod_{j=i+1}^k\left(1-\frac{\alpha_il}{2}\right)
\end{equation*}
holds in probability, where $\kappa'=\min(1,\kappa)$, the constants $M_{G,p}^{(k)}$ and $M_{d,p,2}$ come from \cref{SGFD:lem:EmB2}.
\end{theorem}

\subsection{Convergence for strongly convex objectives}
\label{SGFD:s4:5}

From \cref{SGFD:thm:ABm}, the following theorem could be proved in the same way as the proof of \cref{SGFD:thm:main1}. The only difference is replacing the factor $\beta l$ with $\frac{\beta l}{2}$. 
\begin{theorem}\label{SGFD:thm:main2}
Suppose the conditions of \cref{SGFD:thm:ABm} hold. Then the bound 
\begin{equation*}
  \mathbb{E}[F(x_k)]-F_*\leqslant
  \frac{C_A\Gamma(1+\sigma)}{\Gamma(1+\sigma-\frac{\beta l}{2})}
  \frac{F(x_1)-F_*}{(k+1+\sigma)^{\frac{\beta l}{2}}}
  +\frac{C_B\beta^2}{\beta l-2}\frac{LM_{d,\gamma}}{k+1+\sigma}
\end{equation*}
holds in probability. According to \cref{SGFD:eq:stepsizem}, we have $\beta l>4$, then $\mathbb{E}[F(x_k)]-F_*=\mathcal{O}(1/k)$.
\end{theorem}

Note that for the $(k+1)$th iteration, the entire mean-variance analysis process is only related to the first $k$ iterations. Thus, we have proved \cref{SGFD:eq:CRf} by induction on $k$. 

Now we will prove that the accelerated methods with momentum can achieve a convergence rate $\mathcal{O}(1/k^{1+s})$.
\begin{theorem}\label{SGFD:thm:main3}
Suppose the conditions of \cref{SGFD:thm:ABm2} hold. Then the bound
\begin{equation*}
  \mathbb{E}[F(x_k)]-F_*\leqslant
  \frac{C_A\Gamma(1+\sigma)}{\Gamma(1+\sigma-\frac{\beta l}{2})}
  \frac{F(x_1)-F_*}{(k+1+\sigma)^{\frac{\beta l}{2}}}+
  \frac{C'_B\beta^{2+\kappa'}}{\beta l-2}
  \frac{LM_{d,\gamma}}{(k+1+\sigma)^{1+\kappa'}}
\end{equation*}
holds in probability, where $\kappa'=\min(1,\kappa)$. Further, according to \cref{SGFD:eq:stepsizem}, we have $\beta l>4$, then it holds that $\mathbb{E}[F(x_k)]-F_*=\mathcal{O}(1/k^{1+\kappa'})$.
\end{theorem}
\begin{proof}
Similar to \cref{SGFD:eq:Ak}, we obtain
\begin{align*}
  \sum_{i=1}^k\alpha_i^{2+\kappa'}\!\!\!\prod_{j=i+1}^k\!\!
  \left(1-\frac{\alpha_il}{2}\right)
  =&\sum_{i=1}^k\frac{\beta^{2+\kappa'}}{(i+\sigma)^{2+\kappa'}}
  \prod_{j=i+1}^k\frac{j+\sigma-\frac{\beta l}{2}}{j+\sigma} \\
  =&\beta^{2+\kappa'}\frac{\Gamma(k+1+\sigma-\frac{\beta l}{2})}
  {\Gamma(k+1+\sigma)}\sum_{i=1}^k\frac{1}{(i+\sigma)^{2+\kappa'}}
  \frac{\Gamma(i+1+\sigma)}{\Gamma(i+1+\sigma-\frac{\beta l}{2})}.
\end{align*}
And it follows from \cref{SGFD:lem:GammaRatio} that
\begin{align*}
  \sum_{i=1}^k\frac{1}{(i\!+\!\sigma)^{2+\kappa'}}
  \frac{\Gamma(i\!+\!1\!+\!\sigma)}
  {\Gamma(i\!+\!1\!+\!\sigma\!-\!\frac{\beta l}{2})}
  =&\sum_{i=1}^k\frac{(i\!+\!1\!+\!\sigma\!-\!\frac{\beta l}{2})
  ^{\frac{\beta l}{2}}}{(i\!+\!\sigma)^{2+\kappa'}}
  +\bm{O}\bigg(\sum_{i=1}^k\frac{(i\!+\!1\!+\!\sigma\!-\!\frac{\beta l}{2})
  ^{\frac{\beta l}{2}-1}}{(i\!+\!\sigma)^{2+\kappa'}}\bigg) \\
  =&\frac{C'_k}{\frac{\beta l}{2}-1}
  (k\!+\!1\!+\!\sigma)^{\frac{\beta l}{2}-1-\kappa'}
  +\bm{O}\left((k\!+\!1\!+\!\sigma)^{\frac{\beta l}{2}-2-\kappa'}\right),
\end{align*}
where $C'_k$ tends to $1$ as $k\to\infty$. Together with
\begin{equation*}
  \beta^{2+\kappa'}
  \frac{\Gamma(k+1+\sigma-\frac{\beta l}{2})}{\Gamma(k+1+\sigma)}
  =\beta^{2+\kappa'}(k+1+\sigma)^{-\frac{\beta l}{2}}
  +\bm{O}\left((k+1+\sigma)^{-\frac{\beta l}{2}-1}\right),
\end{equation*}
thus, we finally get
\begin{align*}
  \sum_{i=1}^k\alpha_i^{2+\kappa'}\prod_{j=i+1}^k
  \left(1-\frac{\alpha_il}{2}\right)
  =\frac{C'_k\beta^{2+\kappa'}}{\frac{\beta l}{2}-1}(k+1+\sigma)^{-1-\kappa'}
  +\bm{O}\left((k+1+\sigma)^{-2-\kappa'}\right),
\end{align*}
so the desired result can be proved in the same way as the proof of \cref{SGFD:thm:main1}.
\end{proof}

Therefore, when $0<\kappa\leqslant1$, \cref{SGFD:ass:ASC,SGFD:thm:main3} implies that for every $1\leqslant j\leqslant k$, it holds that
\begin{equation*}
  \mathbb{E}[\|x_j-x_*\|_2^2]\leqslant2l^{-1}(\mathbb{E}[F(x_j)]-F_*)
  =\mathcal{O}(j^{-1-\kappa})=\mathcal{O}(j^{-1-\kappa}),
\end{equation*}
similar to \cref{SGFD:eq:CEd}, we further have 
\begin{equation*}
  \mathbb{E}[\|\delta_{j,k}\|_2^2]=\mathcal{O}(j^{-1-\kappa}),
\end{equation*}
together with \cref{SGFD:ass:Order}, we obtain 
\begin{equation}\label{SGFD:eq:REC2}
  \|\mathbb{E}[\delta_{j,k}]\|_2^2=\mathcal{O}
  \Big(j^{-\kappa}\mathbb{E}[\|\delta_{j,k}\|_2^2]\Big)
  =\mathcal{O}\Big(j^{-\frac{1+3\kappa}{2}}\Big).
\end{equation}
Hence, from \cref{SGFD:eq:REC1,SGFD:eq:REC2}, it is clear that $\|\mathbb{E}[\delta_{j,k}]\|_2^2=\mathcal{O}(j^{-\frac{1+\kappa}{2}})$ implies $\|\mathbb{E}[\delta_{j,k}]\|_2^2= \mathcal{O}(j^{-\frac{1+3\kappa}{2}})$ for every $0<\kappa\leqslant1$, which actually means
\begin{equation*}
  \|\mathbb{E}[\delta_{j,k}]\|_2^2=\mathcal{O}(j^{-1}),
\end{equation*}
and then, we have the final result. 
\begin{corollary}\label{SGFD:cor:main}
Suppose the conditions of \cref{SGFD:thm:main3} hold. Then the bound 
\begin{equation*}
  \mathbb{E}[F(x_k)]-F_*\leqslant
  \frac{C_A\Gamma(1+\sigma)}{\Gamma(1+\sigma-\frac{\beta l}{2})}
  \frac{F(x_1)-F_*}{(k+1+\sigma)^{\frac{\beta l}{2}}}+
  \frac{C'_B\beta^3}{\beta l-2}\frac{LM_{d,\gamma}}{(k+1+\sigma)^2}
\end{equation*}
holds in probability. Further, according to \cref{SGFD:eq:stepsizem}, we have $\beta l>4$, then it holds that $\mathbb{E}[F(x_k)]-F_*=\mathcal{O}(1/k^2)$.
\end{corollary}

\subsection{Convergence for general objectives}

From \cref{SGFD:lem:EmB}, the following lemma could be proved in the same way as the proof of \cref{SGFD:lem:nonconvex}.
\begin{lemma}\label{SGFD:lem:mnonconvex}
Under the conditions of \cref{SGFD:lem:EmB,SGFD:ass:AG}, suppose that the objective $F$ is bounded below by a scalar $F_{\textrm{inf}}<\infty$, and the sequence of iterates $\{x_k\}$ is generated by \cref{SGFD:eq:m2} with a fixed scalar $\gamma\in(0,1)$ and a non-increasing stepsize sequence $\{\alpha_k\}$ taking the form \cref{SGFD:eq:stepsize} and satisfying $\alpha_k\leqslant\frac{1}{LM_{G,p}^{(1)}}$, then
\begin{equation*}
  \sum_{i=1}^\infty\alpha_i\mathbb{E}\big[\|\nabla F(x_k)\|_2^2\big]<\infty,
\end{equation*}
and therefore, $\liminf_{k\to\infty}\mathbb{E}[\|\nabla F(x_k)\|_2^2]=0$.
\end{lemma}

According to \cref{SGFD:thm:momentum,SGFD:lem:mnonconvex}, the following theorem, which guarantees that the expected gradient norms converge to zero for the accelerated method with momentum, could be proved in the same way as the proof of \cref{SGFD:thm:nonconvex}.
\begin{theorem}\label{SGFD:thm:mnonconvex}
Suppose the conditions of \cref{SGFD:lem:mnonconvex} hold. Then
\begin{equation*}
  \lim_{k\to\infty}\mathbb{E}[\|\nabla F(x_k)\|_2^2]=0.
\end{equation*}
\end{theorem}

\section{Conclusions}
\label{SGFD:s5}

In this paper we propose stochastic gradient-free methods and accelerated methods with momentum for solving stochastic optimization problems. Our gradient-free methods maintain the sublinear convergence rate $\mathcal{O}(1/k)$ and the accelerated methods achieve a convergence rate $\mathcal{O}(1/k^2)$ when employing a decaying stepsize $\alpha_k= \mathcal{O}(1/k)$ for the strongly convex objectives with Lipschitz gradients; and all these methods converge to a solution with a zero expected gradient norm when the objective function is nonconvex, twice differentiable and bounded below. Moreover, we provide a mean-variance framework and a theoretical analysis about the inclusion of momentum in stochastic settings. The latter one reveals that the momentum term we used adds a deviation of order $\mathcal{O}(1/k)$ but controls the variance at the order $\mathcal{O}(1/k)$ for the $k$th iteration, and this is why the proposed accelerated methods can achieve a better convergence rate.





\bibliographystyle{siamplain}
\bibliography{MReferences}

\begin{thebibliography}{10}

\bibitem{AgarwalA2013A_Bandit}
{\sc A.~Agarwal, D.~P. Foster, D.~Hsu, S.~M. Kakade, and A.~Rakhlin}, {\em
  Stochastic convex optimization with bandit feedback}, SIAM J. Optim., 23(1)
  (2013), pp.~213--240.

\bibitem{AsiH2019A_SPP}
{\sc H.~Asi and J.~C. Duchi}, {\em Stochastic (approximate) proximal point
  methods: Convergence, optimality, and adaptivity}, SIAM J. Optim., 29(3)
  (2019), pp.~2257--2290.

\bibitem{BengioY1994A_GradientDifficult}
{\sc Y.~Bengio, P.~Simard, and P.~Frasconi}, {\em Learning long-term
  dependencies with gradient descent is difficult}, IEEE Transactions on Neural
  Networks, 5(2) (1994), pp.~157--166.

\bibitem{BottouL2007A_TradeoffsLearning}
{\sc L.~Bottou and O.~Bousquet}, {\em The tradeoffs of large scale learning},
  in Adv. Neural Inf. Process. Syst. 20, J.~C. Platt, D.~Koller, Y.~Singer, and
  S.~T. Roweis, eds., Curran Associates, Red Hook, NY, 2007, pp.~161--168.

\bibitem{BottouL2018R_SGD}
{\sc L.~Bottou, F.~E. Curtis, and J.~Nocedal}, {\em Optimization methods for
  large-scale machine learning}, SIAM Review, 60(2) (2018), pp.~223--311.

\bibitem{ChungKL1954A_SG}
{\sc K.~L. Chung}, {\em On a stochastic approximation method}, Ann. Math.
  Statist., 25 (1954), pp.~463--483.

\bibitem{ConnA2009M_DerivativeFree}
{\sc A.~Conn, K.~Scheinberg, and L.~Vicente}, {\em Introduction to
  derivative-free optimization}, MPSSIAM series on optimization, SIAM,
  Philadelphia, 2009.

\bibitem{DeanJ2012A_SG&DNN}
{\sc J.~Dean, G.~S. Corrado, R.~Monga, K.~Chen, M.~Devin, Q.~V. Le, M.~Z. Mao,
  M.~Ranzato, A.~Senior, P.~Tucker, K.~Yang, and A.~Y. Ng}, {\em Large scale
  distributed deep networks}, in Adv. Neural Inf. Process. Syst. 25, Curran
  Associates, Red Hook, NY, 2012, pp.~1223--1231.

\bibitem{DuchiJ2015A_ZeroOrderCO}
{\sc J.~C. Duchi, M.~I. Jordan, M.~J. Wainwright, and A.~Wibisono}, {\em
  Optimal rates for zero-order convex optimization: The power of two function
  evaluations}, IEEE Trans. Information Theory, 61(5) (2015), pp.~2788--2806.

\bibitem{GorbunovE2019A_DerivativeFree}
{\sc E.~Gorbunov, P.~Dvurechensky, and A.~Gasnikov}, {\em An accelerated method
  for derivative-free smooth stochastic convex optimization}, preprint,
  https://arxiv.org/abs/1802.09022v2 (2019).

\bibitem{HazanE2014A_BanditOptimization}
{\sc E.~E. Hazan and K.~Y. Levy}, {\em Bandit convex optimization: Towards
  tight bounds}, Advances in Neural Information Processing Systems, 1 (2014),
  pp.~784--792.

\bibitem{LeCunY2015A_DeepLearning}
{\sc Y.~LeCun, Y.~Bengio, and G.~Hinton}, {\em Deep learning}, Nature, 521
  (2015), pp.~436--444.

\bibitem{LeenT1993A_momentum}
{\sc T.~K. Leen and G.~B. Orr}, {\em Optimal stochastic search and adaptive
  momentum}, in Advances in Neural Information Processing Systems, vol.~6,
  1993, pp.~477--484.

\bibitem{LojasiewiczS1963A_PolyakGradient}
{\sc S.~{\L}ojasiewicz}, {\em A topological property of real analytic subsets
  (in french)}, Coll. du CNRS, Les \'{e}quations aux d\'{e}riv\'{e}es
  partielles,  (1963), pp.~87--89.

\bibitem{MatyasJ1965A_RandomOptimization}
{\sc J.~Matyas}, {\em Random optimization}, Automation and Remote Control, 26
  (1965), pp.~246--253.

\bibitem{NemirovskiA2009A_StochasticProgramming}
{\sc A.~Nemirovski, A.~Juditsky, G.~Lan, and A.~Shapiro}, {\em Robust
  stochastic approximation approach to stochastic programming}, SIAM J. Optim.,
  19 (2009), pp.~1574--1609.

\bibitem{NesterovY2017A_GradientFree}
{\sc Y.~Nesterov and V.~Spokoiny}, {\em Random gradient-free minimization of
  convex functions}, Found. Comput. Math., 17(2) (2017), pp.~527--566.

\bibitem{PascanuR2013A_DifficultyRNN}
{\sc R.~Pascanu, T.~Mikolov, and Y.~Bengio}, {\em On the difficulty of training
  recurrent neural networks}, in Proceedings of the 30th International
  Conference on International Conference on Machine Learning (ICML), vol.~28,
  2013, pp.~1310--1318.

\bibitem{PolyakB1963A_Gradient}
{\sc B.~T. Polyak}, {\em Gradient methods for minimizing functionals (in
  russian)}, Zh. Vychisl. Mat. Mat. Fiz., 3(4) (1963), pp.~643--653.

\bibitem{PolyakB1964A_momentum}
{\sc B.~T. Polyak}, {\em Some methods of speeding up the convergence of
  iteration methods}, USSR Comput. Math. Math. Phys., 4 (1964), pp.~1--17.

\bibitem{RobbinsH1951A_SG}
{\sc H.~Robbins and S.~Monro}, {\em A stochastic approximation method}, Ann.
  Math. Statist., 22 (1951), pp.~400--407.

\bibitem{SacksJ1958A_SG}
{\sc J.~Sacks}, {\em Asymptotic distribution of stochastic approximation
  procedures}, Ann. Math. Statist., 29(2) (1958), pp.~373--405.

\bibitem{ShwartzS2011A_subgradientSVM}
{\sc S.~Shalev-Shwartz, Y.~Singer, N.~Srebro, and A.~Cotter}, {\em {PEGASOS}:
  Primal estimated sub-gradient solver for {SVM}}, Math. Program., 127 (2011),
  pp.~3--30.

\bibitem{ShamirO2017A_ZeroOrderConvex}
{\sc O.~Shamir}, {\em An optimal algorithm for bandit and zero-order convex
  optimization with two-point feedback}, Journal of Machine Learning Research,
  18 (2017), pp.~1--11.

\bibitem{ShapiroA2009M_StochasticProgramming}
{\sc A.~Shapiro, D.~Dentcheva, and A.~Ruszczy\'{n}ski}, {\em Lectures on
  Stochastic Programming: Modeling and Theory}, MPS-SIAM Ser. Optim., SIAM,
  Philadelphia, PA, 2009.

\bibitem{StichS2013A_ConvexDF}
{\sc S.~U. Stich, C.~L. M\"{u}ller, and B.~G\"{a}rtner}, {\em Optimization of
  convex functions with random pursuit}, SIAM J. Optim., 23(2) (2013),
  pp.~1284--1309.

\bibitem{HintonG2013M_SG_DNN}
{\sc I.~Sutskever, J.~Martens, G.~E. Dahl, and G.~E. Hinton}, {\em On the
  importance of initialization and momentum in deep learning}, in 30th
  International Conference on Machine Learning (ICML), 2013.

\bibitem{TricomiF1951A_GammaRatio}
{\sc F.~G. Tricomi and A.~Erd\'{e}lyi}, {\em The asymptotic expansion of a
  ratio of gamma functions}, Pacific Journal of Mathematics, 1 (1951),
  pp.~133--142.

\bibitem{ZhangT2004A_SG}
{\sc T.~Zhang}, {\em Solving large scale linear prediction problems using
  stochastic gradient descent algorithms}, in Proceedings of the Twenty-First
  International Conference on Machine Learning, ACM Press, New York, 2004.

\bibitem{ZinkevichM2003A_SG}
{\sc M.~Zinkevich}, {\em Online convex programming and generalized
  infinitesimal gradient ascent}, in Proceedings of the Twentieth International
  Conference on Machine Learning, AAAI Press, Palo Alto, CA, 2003,
  pp.~928--935.

\end{thebibliography}
\end{document}